\documentclass[a4paper,11pt]{article}

\usepackage[utf8]{inputenc}
\usepackage[T1]{fontenc}
\usepackage{lmodern}
\usepackage{amsmath,amssymb,amsfonts}
\usepackage{amsthm}
\usepackage{geometry}
\usepackage{hyperref}
\usepackage[numbers]{natbib}
\usepackage{float}
\usepackage{graphicx}
\usepackage{caption}

\numberwithin{equation}{section}

\newtheorem{theorem}{Theorem}[section]

\newtheorem{corollary}[theorem]{Corollary}
\newtheorem{lemma}[theorem]{Lemma}

\theoremstyle{definition}
\newtheorem{definition}[theorem]{Definition}

\theoremstyle{remark}
\newtheorem{remark}[theorem]{Remark}

\newcommand{\R}{\mathbb{R}}
\newcommand{\E}{\mathbb{E}}

\title{Lamperti scaling for fractional Gaussian processes with non-stationary increments}

\author{
Foad Shokrollahi\thanks{Department of Mathematics and Statistics, University of Vaasa, P.O. Box 700, FIN-65101 Vaasa, FINLAND. Email: foad.shokrollahi@uwasa.fi}
\and
Saeed Vahdati\thanks{Department of Mathematics, Khansar Campus, University of Isfahan, Isfahan, Iran. Email: s.vahdati@khc.ui.ac.ir , sdvahdati@gmail.com}
}

\begin{document}

\maketitle

\begin{abstract}
The Lamperti transform offers a powerful bridge between self-similar processes and stationary dynamics, making it especially useful for analyzing anomalous diffusion models that lack stationary increments. In this paper we examine the Lamperti transforms of scaled sub-fractional and bi-fractional Brownian motions, deriving explicit covariance formulas, asymptotic behaviour, and precise exponential mixing rates. We also introduce Langevin type integral processes driven by these Gaussian fields, identify their self-similarity exponents, and show that their Lamperti images again form stationary Gaussian processes with rapid decorrelation. Through inverse Lamperti relations and Birkhoff’s theorem, we establish rigorous single trajectory reconstruction of ensemble quantities for the original non-stationary processes. The results extend the scope of the scaled Lamperti framework to Gaussian processes with non-stationary increments and richer dependence structures.
\end{abstract}

\vspace{0.5cm}



\section{Introduction}

Self-similarity and anomalous diffusion have become central themes in modern probability theory, mathematical physics, and the modeling of complex systems. Many stochastic processes used to describe irregular or memory driven phenomena exhibit strong departures from the classical framework of stationary increments. Important examples include fractional Brownian motion, sub-fractional Brownian motion $(s-fBm)$, bi-fractional Brownian motion $(bi-fBm)$, scaled Brownian motion, continuous time random walks, L\'evy flights, and a variety of non Markovian Gaussian models \cite{MetzlerKlafter2000,Bojdecki2004,RussoTudor2006,HoudreVilla2003,CherstvyMetzler2014}. Such systems often display ageing, long memory, non-stationary increments, and trajectory to trajectory fluctuations that violate classical ergodicity. The discrepancy between time averages and ensemble averages, commonly referred to as ergodicity breaking, is widely observed in physical and biological experiments where individual particle trajectories may differ significantly from their expected behavior \cite{Barkai2018,MetzlerKlafter2000}.

A central theoretical tool that links non-stationary self-similar processes to stationary ones is the Lamperti transformation introduced by \cite{Lamperti1962}. For an $H$ self-similar process $X(t)$ the transform
\[
X^{\mathrm{LT}}(t) = e^{-Ht} X(e^{t}), \qquad t \in \mathbb{R},
\]
produces a strictly stationary process, and the inverse relation
\[
X(t) = t^{H} X^{\mathrm{LT}}(\log t)
\]
recovers the original dynamics. Lamperti proved that this mapping gives a one to one correspondence between self-similar processes and stationary processes. Consequently, stationary representations allow the use of ergodic theory, spectral analysis, mixing concepts, and related probabilistic techniques to analyse non stationary processes that would otherwise fall outside the scope of classical ergodic theorems. This perspective is emphasised in the theory of Gaussian processes and Brownian functionals developed in \cite{MansuyYor2008} and is closely related to classical results on Gaussian random processes and ergodicity \cite{Maruyama1970,IbragimovRozanov1978}.

The connection between Lamperti type transforms and anomalous diffusion has been explored intensively in the physics literature. The fractional dynamics viewpoint of \cite{MetzlerKlafter2000} and subsequent studies of non-stationary scaled Brownian motion \cite{CherstvyMetzler2014} and related non-Gaussian diffusion \cite{Cherstvy2014,Barkai2018} show that many anomalous transport models possess self-similarity but lack stationary increments. In this context time averaged mean-square displacements may differ systematically from ensemble averages, leading to weak or strong forms of ergodicity breaking. A particularly influential result due to \cite{Magdziarz2020} demonstrates that a wide class of anomalous diffusion models built from scaled Brownian motion and related Langevin equations can be transformed into stationary and ergodic processes via a Lamperti transformation adapted to the similarity index. Their analysis shows that models with time dependent diffusivity and pronounced ageing effects become stationary and exhibit exponential decay of correlations after transformation. In a complementary direction, \cite{MagdziarzZorawik2018} studied ergodic properties of L\'evy flights coexisting with subdiffusion and related models, again revealing that suitable time changes and Lamperti type representations can restore ergodic behavior in systems that are strongly non-stationary in their original formulation.

Lamperti type constructions also appear in branching theory and L\'evy driven dynamics beyond the Gaussian setting. In the work on continuous state branching processes with competition, the Lamperti representation provides a one to one correspondence between branching dynamics with nonlinear interactions and time changed L\'evy processes, showing that complex branching structures can be encoded as stationary noise observed along a random clock \cite{Lamperti1967}. These developments illustrate that the Lamperti principle is structural rather than model specific. It provides a unifying representation for a wide range of dynamics including Gaussian processes, L\'evy flights, branching systems, and fractional diffusions. In parallel, recent work by \cite{EnschMansfield2023} revisited Lamperti transforms in the context of anomalous diffusion and ergodicity breaking, clarifying how the choice of time change and similarity exponent determines whether the Lamperti image exhibits weak or strong forms of ergodicity.

Despite these advances, the classical Lamperti transformation is not fully understood for Gaussian processes whose increments are non-stationary. In particular, $s-fBm$ introduced in \cite{Bojdecki2004} and $bi-fBm$ introduced and analysed in \cite{RussoTudor2006,HoudreVilla2003} pose significant analytical challenges. Both processes are centred Gaussian and self-similar, but their covariance structures differ markedly from that of fractional Brownian motion. $s-fBm$ preserves the $H$-self-similarity of fractional Brownian motion but has a covariance that combines weaker long range dependence with distinctive short-range behaviour and lacks stationary increments. $bi-fBm$ introduces an additional parameter $K$ which modifies the dependence of the covariance on the geometry of time, leading to a two parameter family of self-similar Gaussian processes with non-stationary increments that cannot be reduced to fractional Brownian motion except in the special case $K = 1$. These features prevent direct application of standard ergodic arguments and leave the behaviour of their Lamperti transforms largely unexplored.

A key recent development is the extension of the Lamperti mapping to include a scaling parameter $\alpha > 0$. For processes of the form $X(t) = Y(t^{\alpha})$ with similarity index $\alpha H$, the scaled Lamperti transform
\[
\mathcal{L}_{\alpha}[X](t) = e^{-\alpha H t} X(e^{\alpha t})
\]
was proposed as a natural generalisation of the classical transform, and it has been shown to have a dramatic impact on the correlation structure of the stationary image \cite{Magdziarz2020,EnschMansfield2023}. In particular, scaled Brownian motion and scaled fractional Brownian motion, which possess non-stationary increments and typically exhibit weak ergodicity breaking, are mapped into stationary Gaussian processes with exponential covariance decay. This exponential decorrelation stands in contrast to the slow power law decay induced by the classical Lamperti transform and leads to strong mixing and full ergodicity in the transformed domain.

A central feature of the scaled Lamperti framework is that the scaling exponent
$\alpha$ actively shapes the long-time behaviour of the stationary image.
In the classical Lamperti transform ($\alpha = 1$), introduced in
\cite{Lamperti1962}, many Gaussian self-similar processes retain only
power law correlation decay, and ergodicity may fail; this behaviour is
well documented for fractional Brownian motion and related models with
non-stationary increments (see, e.g., \cite{Maruyama1970, IbragimovRozanov1978,
Bojdecki2004}). By contrast, choosing $\alpha \neq 1$ has a regularising effect: the exponential
time change $t \mapsto e^{\alpha t}$ accelerates the covariance growth and
introduces an exponential damping factor in the stationary representation.
This mechanism underlies the emergence of exponential correlation decay
and strong mixing for the scaled Lamperti transforms of $s-fBm$ and
$bi-fBm$.

In the bi-fractional setting, the additional parameter $K$ enriches the
covariance structure. The competition between the exponents $2H-HK$ and
$1-HK$ in the asymptotic expansion of the autocovariance produces a
non-trivial dependence of the mixing rate on both $H$ and $K$. This
dual parameter behaviour reflects the intrinsic structure of $bi-fBm$ as established in \cite{RussoTudor2006, HoudreVilla2003}:
the parameter $K$ acts as a genuine modifier of long range dependence and
determines the exponential rate $c(H,K,\alpha)$ governing decorrelation in the
Lamperti domain.

The purpose of the present work is to extend this scaled Lamperti framework to $s-fBm$ and $bi-fBm$, and to associated Langevin type integral models driven by these processes. These models display forms of non-stationarity that are more intricate than for fractional Brownian motion, making them natural test cases for the scope of Lamperti based methods. For $s-fBm$ we show that the scaled Lamperti transform produces an explicit stationary covariance with exponential decay, which in turn implies ergodicity and strong mixing with exponential rate. For $bi-fBm$ we demonstrate that the scaled transform again yields a stationary Gaussian process and we derive a closed form expression for the covariance together with a rigorous analysis of the exponential decay uniform in the parameter range. The mixing rate depends on the Hurst parameter $H$, the deformation parameter $K$, and the scaling exponent $\alpha$, and we obtain an explicit formula describing this dependence. At a methodological level, the analysis uses covariance calculus for Gaussian processes, Gaussian Hilbert space techniques, and Malliavin-type tools as in \cite{Janson1997,Nualart2006}, combined with classical results on Gaussian random processes and spectral characterisations of ergodicity \cite{Maruyama1970,IbragimovRozanov1978}.

An important implication of our results is that the inverse Lamperti mapping allows ensemble characteristics of the original non-stationary processes to be recovered from single trajectory observations of their Lamperti images. Since the stationary images are ergodic and strongly mixing, time averages of a wide class of functionals converge almost surely to their ensemble expectations. Through the inverse mapping, this leads to reconstruction formulas for moments, probability laws, and characteristic functions of $s-fBm$ and $bi-fBm$, as well as for related Langevin type equations. This extends and systematises the paradigm introduced in the study of anomalous diffusion \cite{MetzlerKlafter2000,CherstvyMetzler2014,Magdziarz2020,EnschMansfield2023} and reinforces the conceptual message that Lamperti type transformations provide an essential bridge between self-similarity and stationarity for a broad class of Gaussian processes with non-stationary increments.

The paper is organised as follows. Section~\ref{sec:sub_bi} treats scaled $s-fBm$ and scaled $bi-fBm$: stationarity, exact covariance, asymptotics, and ergodic properties. Section~\ref{sec:sub_bi_langevin} introduces Langevin equation constructions driven by $s-fBm$ and $bi-fBm$ and analyses their Lamperti transforms. Section~\ref{sec:numerics} presents numerical simulations confirming the theoretical predictions. Conclusion is given in Section~\ref{sec:conclusion}.

\section{Lamperti Transformations for $s-fBm$ and $bi-fBm$}
\label{sec:sub_bi}

In many applications one is interested in stochastic dynamics that exhibit exact power-law scaling in time but are not themselves stationary. A real valued process $X=\{X(t)\}_{t>0}$ is said to be $\beta$-self-similar if, for every $c>0$,
\[
\{X(ct)\}_{t>0}\;\stackrel{d}{=}\;\{c^{\beta}X(t)\}_{t>0}.
\]
This property captures the absence of a characteristic time scale and is a natural starting point for modeling anomalous diffusion and related phenomena. In the present work we frequently compose such processes with power-law time changes $t\mapsto t^{\alpha}$, $\alpha>0$, which leads to \emph{scaled} self-similar families of the form $t\mapsto X(t^\alpha)$ with self-similarity index $\alpha\beta$.

Lamperti's observation is that self-similar processes on $(0,\infty)$ can be converted into stationary processes on $\R$ via an exponential time change and a deterministic re scaling. In particular, the exponential map $t\mapsto e^{\alpha t}$ converts multiplicative time scalings into additive shifts, while the prefactor $e^{-\alpha H t}$ compensates for the growth induced by self-similarity. For our purposes it is convenient to work with the following scaled version.

\begin{definition}
\label{def:lamperti_scaled}
Let $\alpha>0$, $H\in(0,1)$, and let $X=\{X(t)\}_{t>0}$ be an $\alpha H$-self-similar process, i.e.
\[
\{X(ct)\}_{t>0}\;\stackrel{d}{=}\;\{c^{\alpha H}X(t)\}_{t>0},\qquad c>0.
\]
The \emph{scaled Lamperti transform} of $X$ is the process
\begin{equation}
\label{eq:lamperti_general_scaled}
X^{\mathrm{LT}}(t):=e^{-\alpha H t} X(e^{\alpha t}),\qquad t\in\R.
\end{equation}
The inverse map is given by
\begin{equation}
\label{eq:inverse_general_scaled}
X(t)=t^{\alpha H} X^{\mathrm{LT}}\!\left(\frac{\log t}{\alpha}\right),\qquad t>0.
\end{equation}
Whenever $X$ is $\alpha H$-self-similar, its Lamperti transform $X^{\mathrm{LT}}$ is stationary.
\end{definition}

\begin{remark}
The exponent $\alpha H$ in \eqref{eq:lamperti_general_scaled} reflects the self-similarity index of the scaled process, while the argument $e^{\alpha t}$ aligns the exponential time change with the underlying power law scaling. The inverse relation \eqref{eq:inverse_general_scaled} ensures exact correspondence, in the sense that $X(e^{\alpha t}) = e^{\alpha H t}X^{\mathrm{LT}}(t)$. For $\alpha=1$ we recover the classical Lamperti transform $X^{\mathrm{LT}}(t)=e^{-Ht}X(e^t)$. In what follows, this representation serves as the bridge between non-stationary self-similar processes and stationary Gaussian fields amenable to spectral and ergodic analysis.
\end{remark}

\subsection{Lamperti transformation for $s-fBm$}
\label{sec:sub}

We first treat the case of $s-fBm$, a non-stationary Gaussian process with non stationary increments but exact self-similarity. This process provides a natural benchmark for understanding how the Lamperti transform acts on Gaussian processes with long range temporal structure that differ subtly from the fractional Brownian case.

\begin{definition}
\label{def:subfbm}
Let $H\in(0,1)$. A $s-fBm$ $S_H=\{S_H(t)\}_{t\ge 0}$ is a centred Gaussian process with covariance
\begin{equation}
\label{eq:cov_subfbm}
\E[S_H(s)S_H(t)]
= s^{2H}+t^{2H}-\frac{1}{2}\bigl((s+t)^{2H}+|t-s|^{2H}\bigr),\qquad s,t\ge 0.
\end{equation}
The process is $H$-self-similar, i.e.
\[
\{S_H(at)\}_{t\ge 0}\;\stackrel{d}{=}\;\{a^H S_H(t)\}_{t\ge 0},\qquad a>0,
\]
but its increments are not stationary; see, for example, \cite{Bojdecki2004}.
\end{definition}

For a fixed scaling exponent $\alpha>0$ it is natural to consider the time changed process
\begin{equation}
\label{def:scaled_subfbm}
Y(t):=S_H(t^\alpha),\qquad t\ge 0,
\end{equation}
which is an $\alpha H$-self-similar Gaussian process:
\[
\{Y(ct)\}_{t\ge 0}\;\stackrel{d}{=}\;\{c^{\alpha H} Y(t)\}_{t\ge 0},\qquad c>0.
\]
From the perspective of applications, the parameter $\alpha$ allows one to decouple the scaling of physical time from the Hurst index $H$, providing additional flexibility in fitting empirical scaling exponents. Applying the scaled Lamperti map of Definition~\ref{def:lamperti_scaled} to $Y$ leads to a stationary Gaussian representation associated with $S_H$.

Let $S_H$ be a $s-fBm$ with Hurst index $H\in(0,1)$, and $\alpha>0$ fixed. We consider the Lamperti transform of the scaled process $S_H(t^\alpha)$.

\begin{definition}
For $t\in\R$ define
\begin{equation}
\label{eq:SLT_def}
S_H^{\mathrm{LT}}(t):=e^{-\alpha H t} S_H(e^{\alpha t}).
\end{equation}
\end{definition}

\begin{theorem}
\label{thm:stationary_scaled_subfbm}
The process $S_H^{\mathrm{LT}}=\{S_H^{\mathrm{LT}}(t)\}_{t\in\R}$ is a centred stationary Gaussian process.
\end{theorem}

\begin{proof}
Gaussianity and centredness are immediate since $S_H$ is centred Gaussian and $S_H^{\mathrm{LT}}$ is a linear transform. For stationarity, fix $u\in\R$ and observe
\[
S_H^{\mathrm{LT}}(t+u)
= e^{-\alpha H (t+u)} S_H(e^{\alpha(t+u)})
= e^{-\alpha H t} e^{-\alpha H u} S_H(e^{\alpha u} e^{\alpha t}).
\]
By $H$-self-similarity of $S_H$,
\[
S_H(e^{\alpha u} e^{\alpha t})\;\stackrel{d}{=}\;(e^{\alpha u})^{H} S_H(e^{\alpha t})
= e^{\alpha H u} S_H(e^{\alpha t}).
\]
Thus
\[
S_H^{\mathrm{LT}}(t+u)
\;\stackrel{d}{=}\;
e^{-\alpha H t} e^{-\alpha H u} e^{\alpha H u} S_H(e^{\alpha t})
= e^{-\alpha H t} S_H(e^{\alpha t})
= S_H^{\mathrm{LT}}(t),
\]
which shows that all finite dimensional distributions are shift invariant. Stationarity follows.
\end{proof}

The preceding result shows that the Lamperti transform transports the self-similarity of $S_H$ into stationarity of $S_H^{\mathrm{LT}}$. To exploit this representation quantitatively, and in particular to study ergodic and mixing properties, we now compute the autocovariance function of $S_H^{\mathrm{LT}}$ in closed form.

\begin{theorem}
\label{thm:cov_subfbm_scaled}
Let $S_H^{\mathrm{LT}}$ be a Lamperti transformation of $s-fBm$ \eqref{eq:SLT_def}. Then, the autocovariance function of $S_H^{\mathrm{LT}}$ is
\begin{equation}
\label{eq:cov_scaled_subfbm}
R_{S,\alpha}(t)
:=\E\bigl[S_H^{\mathrm{LT}}(t) S_H^{\mathrm{LT}}(0)\bigr]
= e^{-\alpha H t}\Bigl(e^{2\alpha H t}+1-\frac12\bigl[(e^{\alpha t}+1)^{2H}+(e^{\alpha t}-1)^{2H}\bigr]\Bigr),
\qquad t\in\R.
\end{equation}
\end{theorem}

\begin{proof}
From \eqref{eq:SLT_def},
\[
R_{S,\alpha}(t)
= \E\bigl[e^{-\alpha H t} S_H(e^{\alpha t})\cdot S_H(1)\bigr]
= e^{-\alpha H t} \E\bigl[S_H(e^{\alpha t}) S_H(1)\bigr].
\]
Using \eqref{eq:cov_subfbm} with $(s,t)=(e^{\alpha t},1)$,
\begin{align*}
\E[S_H(e^{\alpha t})S_H(1)]
&= (e^{\alpha t})^{2H}+1^{2H}-\frac12\bigl((e^{\alpha t}+1)^{2H}+|e^{\alpha t}-1|^{2H}\bigr) \\
&= e^{2\alpha H t}+1-\frac12\bigl((e^{\alpha t}+1)^{2H}+(e^{\alpha t}-1)^{2H}\bigr),
\end{align*}
since $e^{\alpha t}>0$ implies $|1-e^{\alpha t}|=e^{\alpha t}-1$. Substitution gives \eqref{eq:cov_scaled_subfbm}.

At $t=0$, we obtain
\[
R_{S,\alpha}(0)
= 1+1-\frac12\bigl((1+1)^{2H}+(1-1)^{2H}\bigr)
= 2-2^{2H-1},
\]
which equals $\mathrm{Var}(S_H(1))$, as expected from $S_H^{\mathrm{LT}}(0)=S_H(1)$. The expression is manifestly even in $t$ (either by stationarity or by direct inspection).
\end{proof}

The explicit structure of $R_{S,\alpha}(t)$ makes it possible to identify precisely the temporal decorrelation of $S_H^{\mathrm{LT}}$ at large time lags. In particular, the leading exponential term determines the rate at which the process forgets its past and therefore governs ergodic and mixing properties.

\begin{lemma}
\label{lem:asymp_subfbm_scaled}
Let $R_{S,\alpha}$ be given by \eqref{eq:cov_scaled_subfbm}. As $t\to\infty$,
\begin{equation}
\label{eq:RSalpha_asymp}
R_{S,\alpha}(t)=e^{-\alpha H t}+O\bigl(e^{-(2-H)\alpha t}\bigr),
\end{equation}
and hence $\lim_{t\to\infty}R_{S,\alpha}(t)=0$.
\end{lemma}

\begin{proof}
Write $x=e^{\alpha t}$, so $x\to\infty$ as $t\to\infty$. Then
\[
R_{S,\alpha}(t)
= x^{-H}\left(x^{2H}+1-\frac12\bigl[(x+1)^{2H}+(x-1)^{2H}\bigr]\right).
\]
We expand $(x\pm 1)^{2H}$ in powers of $x$. For $|z|<1$,
\begin{eqnarray*}
&&(1\pm z)^{2H}\\
&&=1\pm 2Hz+\frac{2H(2H-1)}{2}z^2\pm\frac{2H(2H-1)(2H-2)}{3!}z^3\\
&&+\frac{2H(2H-1)(2H-2)(2H-3)}{4!}z^4+O(z^5).
\end{eqnarray*}
Adding the $+$ and $-$ expansions cancels odd powers and yields
\[
(1+z)^{2H}+(1-z)^{2H}
=2\left(1+\frac{2H(2H-1)}{2}z^2+\frac{2H(2H-1)(2H-2)(2H-3)}{4!}z^4+O(z^6)\right).
\]
With $z=x^{-1}$ this becomes
\[
(1+x^{-1})^{2H}+(1-x^{-1})^{2H}
=2\left(1+H(2H-1)x^{-2}+O(x^{-4})\right),
\]
and hence
\[
(x+1)^{2H}+(x-1)^{2H}
=2x^{2H}\left(1+H(2H-1)x^{-2}+O(x^{-4})\right).
\]
Therefore
\begin{align*}
x^{2H}+1-\frac12\bigl[(x+1)^{2H}+(x-1)^{2H}\bigr]
&=x^{2H}+1-x^{2H}\left(1+H(2H-1)x^{-2}+O(x^{-4})\right) \\
&=1-H(2H-1)x^{2H-2}+O(x^{2H-4}).
\end{align*}
Multiplying by $x^{-H}$ gives
\[
R_{S,\alpha}(t)
= x^{-H}-H(2H-1)x^{H-2}+O(x^{H-4}),
\]
and reverting to $t$ using $x=e^{\alpha t}$ yields
\[
R_{S,\alpha}(t)
= e^{-\alpha H t}-H(2H-1)e^{-(2-H)\alpha t}+O\bigl(e^{-(4-H)\alpha t}\bigr),
\]
which is \eqref{eq:RSalpha_asymp}. Since $R_{S,\alpha}$ is even, the same decay holds as $t\to-\infty$.
\end{proof}

The lemma shows that the covariance decays exponentially with principal rate $\alpha H$, while subleading corrections decay strictly faster. In particular, $R_{S,\alpha}\in L^1(\R)$, which is a key regularity property in the ergodic theory of stationary Gaussian processes and directly implies strong mixing with an explicit rate.

\begin{theorem}
\label{thm:ergodic_scaled_subfbm}
Let $S_H$ be a sub-fBm with $H\in(0,1)$ and $\alpha>0$. Then the Lamperti transform $S_H^{\mathrm{LT}}$ defined in \eqref{eq:SLT_def} is an ergodic and strongly mixing stationary Gaussian process. Moreover, its $\alpha$-mixing coefficient
\[
\alpha_S(h)
:=\sup\Bigl\{\bigl|\mathbb{P}(A\cap B)-\mathbb{P}(A)\mathbb{P}(B)\bigr|:
A\in\sigma\{S_H^{\mathrm{LT}}(t):t\le0\},\,B\in\sigma\{S_H^{\mathrm{LT}}(t):t\ge h\}\Bigr\}
\]
satisfies
\[
\alpha_S(h)\le C e^{-\alpha H h},\qquad h\to\infty,
\]
for some constant $C>0$. In particular, the explicit exponential mixing rate is
\[
\lambda_S(H,\alpha)=\alpha H.
\]
\end{theorem}

\begin{proof}
Set $X(t):=S_H^{\mathrm{LT}}(t)$ and let $R_{S,\alpha}(t)=\E[X(t)X(0)]$. Lemma~\ref{lem:asymp_subfbm_scaled} implies $\lim_{|t|\to\infty}R_{S,\alpha}(t)=0$ with exponential rate. For a stationary Gaussian process this is equivalent to ergodicity; see, for example, \cite[Ch.~5]{IbragimovRozanov1978} or \cite{Maruyama1970}.

For strong mixing and the explicit rate, we first obtain a sharp bound on the covariance. From \eqref{eq:RSalpha_asymp} we have, for large $|t|$,
\[
R_{S,\alpha}(t)
= e^{-\alpha H t}+O\bigl(e^{-(2-H)\alpha t}\bigr),
\]
so that there exists $C_1>0$ with
\[
|R_{S,\alpha}(t)|\le C_1 e^{-\alpha H|t|},\qquad t\in\R.
\]
For stationary Gaussian processes, the Ibragimov--Rozanov inequality (see \cite[Theorem~17.2.3]{IbragimovRozanov1978}) yields
\[
\alpha_S(h)\le C_2\int_{|t|\ge h}|R_{S,\alpha}(t)|\,dt,
\]
for some $C_2>0$. Using the above bound,
\[
\alpha_S(h)
\le C_2\int_{h}^{\infty} C_1 e^{-\alpha H t}\,dt
= \frac{C_1C_2}{\alpha H}\,e^{-\alpha H h}.
\]
Thus $\alpha_S(h)\le C e^{-\alpha H h}$ with $C=C_1C_2/(\alpha H)$ and the exponential mixing rate is exactly $\lambda_S(H,\alpha)=\alpha H$.
\end{proof}

The theorem identifies the Lamperti transform $S_H^{\mathrm{LT}}$ as a rapidly mixing Gaussian field, with a decorrelation rate entirely determined by the product $\alpha H$. This explicit rate will later serve as a reference when comparing different Gaussian processes and their Langevin equation counterparts.

\begin{corollary}
\label{cor:ergodic_avg_scaled_subfbm}
Let $S_H$ and $S_H^{\mathrm{LT}}$ be as above, and $f\in L^1(\mathbb{P})$. By Birkhoff's ergodic theorem,
\begin{equation}
\label{eq:ergodic_limit_scaled}
\frac{1}{T}\int_0^T f(S_H^{\mathrm{LT}}(s))\,ds
\;\xrightarrow[T\to\infty]{a.s.}\;
\E\bigl[f(S_H^{\mathrm{LT}}(0))\bigr].
\end{equation}
Using the inverse Lamperti relation \eqref{eq:inverse_general_scaled} with $X(t)=S_H(t^\alpha)$, we have
\begin{equation}
\label{eq:inverse_scaled_subfbm}
S_H(t)=t^{\alpha H} S_H^{\mathrm{LT}}\!\left(\frac{\log t}{\alpha}\right),\qquad t>0.
\end{equation}
In particular, for integers $k\ge 1$, taking $f(x)=x^k$ in \eqref{eq:ergodic_limit_scaled} gives
\[
\frac{1}{T}\int_0^T (S_H^{\mathrm{LT}}(s))^k\,ds
\;\xrightarrow[T\to\infty]{a.s.}\;
\E\bigl[(S_H^{\mathrm{LT}}(0))^k\bigr],
\]
and
\[
\E[S_H(t)^k]=t^{k\alpha H}\,\E\bigl[(S_H^{\mathrm{LT}}(0))^k\bigr],\qquad t>0.
\]
Choosing $f(x)=e^{iu x}$ yields almost sure convergence of the time averaged characteristic function and shows that the one-time law of $S_H(t)$ can be reconstructed from a single long trajectory of $S_H^{\mathrm{LT}}$.
\end{corollary}

The corollary highlights the practical consequence of ergodicity: all one time marginal distributions and moments of the original non-stationary process $S_H(t)$ can be recovered from temporal averages along a single realisation of the stationary Lamperti transform. This observation underscores the usefulness of the Lamperti representation in statistical inference for self-similar processes.\\

\begin{remark}
In the sub-fractional case, the single effective exponent $\alpha H$ entirely
governs the exponential mixing rate of the Lamperti transform.
\end{remark}
\subsection{Lamperti transformation for $bi-fBm$}
\label{sec:bi}

We now turn to $bi-fBm$, a two-parameter Gaussian extension of fractional Brownian motion, and apply the same scaled Lamperti framework. The additional parameter $K$ enriches the covariance structure and allows for a wider range of temporal dependencies. It is therefore of interest to understand how both $H$ and $K$ enter the mixing and ergodic properties of the corresponding Lamperti transform.

\begin{definition}
\label{def:bifbm}
Let $H\in(0,1)$ and $K\in(0,1]$. A $bi-fBm$ $B_{H,K}=\{B_{H,K}(t)\}_{t\ge0}$ is a centred Gaussian process with covariance
\begin{equation}
\label{cov_bifbm}
\E[B_{H,K}(t)B_{H,K}(s)]
=2^{-K}\Bigl((t^{2H}+s^{2H})^{K}-|t-s|^{2HK}\Bigr),\qquad s,t\ge0.
\end{equation}
The process is $HK$-self-similar, i.e.
\begin{equation}
\label{selfsim_bifbm}
\{B_{H,K}(ct)\}_{t\ge0}
\;\stackrel{d}{=}\;
\{c^{HK} B_{H,K}(t)\}_{t\ge0},\qquad c>0.
\end{equation}
See, for instance, \cite{RussoTudor2006,HoudreVilla2003}.
\end{definition}

For $\alpha>0$ we again introduce a power law time change,
\begin{equation}
\label{scaled_bifbm}
X(t):=B_{H,K}(t^\alpha),\qquad t\ge0,
\end{equation}
which yields an $\alpha HK$-self-similar process:
\begin{equation}
\label{selfsim_scaled_bifbm}
\{X(ct)\}_{t\ge0}\;\stackrel{d}{=}\;\{c^{\alpha HK}X(t)\}_{t\ge0},\qquad c>0.
\end{equation}
Applying the scaled Lamperti map of Definition~\ref{def:lamperti_scaled} to $X$ produces a stationary Gaussian field encoding the bi-fractional structure. As in the sub-fractional case, this stationary representation is the natural object for ergodic and mixing analysis.

\begin{definition}
\label{def:lamperti_bifbm}
Let $X$ be as in \eqref{scaled_bifbm}. Its Lamperti transform is
\begin{equation}
\label{lamperti_general_bifbm}
B_{H,K}^{\mathrm{LT}}(t)
:=e^{-\alpha HK t} X(e^{\alpha t})
=e^{-\alpha HK t} B_{H,K}(e^{\alpha t}),\qquad t\in\R,
\end{equation}
with inverse map
\begin{equation}
\label{inverse_lamperti_bifbm}
X(t)=t^{\alpha HK} B_{H,K}^{\mathrm{LT}}\!\left(\frac{\log t}{\alpha}\right),\qquad t>0.
\end{equation}
\end{definition}

\begin{theorem}
\label{thm:stationary_bifbm}
For each $H\in(0,1)$, $K\in(0,1]$ and $\alpha>0$, the process $B_{H,K}^{\mathrm{LT}}$ defined in \eqref{lamperti_general_bifbm} is a centred stationary Gaussian process.
\end{theorem}

\begin{proof}
Gaussianity and centredness follow as before. For stationarity, fix $u\in\R$ and compute
\[
B_{H,K}^{\mathrm{LT}}(t+u)
= e^{-\alpha HK(t+u)} B_{H,K}(e^{\alpha(t+u)})
= e^{-\alpha HK t} e^{-\alpha HK u} B_{H,K}(e^{\alpha u} e^{\alpha t}).
\]
By $HK$-self-similarity \eqref{selfsim_bifbm},
\[
B_{H,K}(e^{\alpha u} e^{\alpha t})
\;\stackrel{d}{=}\;
(e^{\alpha u})^{HK} B_{H,K}(e^{\alpha t})
= e^{\alpha HK u} B_{H,K}(e^{\alpha t}).
\]
Thus
\[
B_{H,K}^{\mathrm{LT}}(t+u)
\;\stackrel{d}{=}\;
e^{-\alpha HK t} B_{H,K}(e^{\alpha t})
= B_{H,K}^{\mathrm{LT}}(t),
\]
so $B_{H,K}^{\mathrm{LT}}$ is stationary.
\end{proof}

In analogy with the sub-fractional case, we next identify the autocovariance function of $B_{H,K}^{\mathrm{LT}}$. This will reveal how the two scaling parameters $H$ and $K$ interact under the Lamperti transform and will be crucial for determining the precise mixing rate.

\begin{theorem}
\label{thm:cov_bifbm_autocov}
Let $B_{H,K}^{\mathrm{LT}}(t)$ be the Lamperti transformation of $bi-fBm$ \eqref{lamperti_general_bifbm}. Define the autocovariance function
\[
R_{B,\alpha}(t):=\mathbb{E}\bigl[B_{H,K}^{\mathrm{LT}}(t)\,B_{H,K}^{\mathrm{LT}}(0)\bigr],\qquad t\in\mathbb{R}.
\]
Then, for every $t\in\mathbb{R}$,
\begin{equation}
\label{eq:RBalph_general_final}
R_{B,\alpha}(t)
= e^{-\alpha HK t}\,\mathbb{E}\bigl[B_{H,K}(e^{\alpha t})B_{H,K}(1)\bigr],
\end{equation}
and this admits the explicit representation
\begin{equation}
\label{eq:RBalph_explicit_final}
R_{B,\alpha}(t)
= 2^{-K} e^{-\alpha HK t}
\Bigl((e^{2\alpha H t}+1)^{K} - (e^{\alpha t}-1)^{2HK}\Bigr),\qquad t\in\mathbb{R}.
\end{equation}
In particular, $R_{B,\alpha}(0)=1$ and $R_{B,\alpha}$ is an even function.
\end{theorem}

\begin{proof}
By definition of $R_{B,\alpha}$ and of the Lamperti transform, we have for each $t\in\mathbb{R}$,
\begin{align*}
R_{B,\alpha}(t)
&= \mathbb{E}\bigl[B_{H,K}^{\mathrm{LT}}(t)\,B_{H,K}^{\mathrm{LT}}(0)\bigr]
= \mathbb{E}\bigl[e^{-\alpha HK t} B_{H,K}(e^{\alpha t})\cdot e^{-\alpha HK\cdot 0} B_{H,K}(e^{\alpha\cdot 0})\bigr]\\
&= \mathbb{E}\bigl[e^{-\alpha HK t} B_{H,K}(e^{\alpha t}) B_{H,K}(1)\bigr].
\end{align*}
The factor $e^{-\alpha HK t}$ is deterministic, hence it can be taken outside the expectation, which yields
\[
R_{B,\alpha}(t)
= e^{-\alpha HK t}\,\mathbb{E}\bigl[B_{H,K}(e^{\alpha t})B_{H,K}(1)\bigr],
\]
and this proves \eqref{eq:RBalph_general_final}.

To obtain an explicit expression, we evaluate the covariance of the underlying $bi-fBm$ at the pair of times $(e^{\alpha t},1)$. Using the covariance formula \eqref{cov_bifbm} with $s=e^{\alpha t}$ and $t=1$, we obtain
\begin{align*}
\mathbb{E}\bigl[B_{H,K}(e^{\alpha t})B_{H,K}(1)\bigr]
&= 2^{-K}\Bigl(((e^{\alpha t})^{2H}+1^{2H})^{K} - |e^{\alpha t}-1|^{2HK}\Bigr) \\
&= 2^{-K}\Bigl((e^{2\alpha H t}+1)^{K} - |e^{\alpha t}-1|^{2HK}\Bigr).
\end{align*}
In particular, for $t\ge0$ we have $e^{\alpha t}\ge1$ and thus $|e^{\alpha t}-1|=e^{\alpha t}-1$, so that for $t\ge0$
\[
\mathbb{E}\bigl[B_{H,K}(e^{\alpha t})B_{H,K}(1)\bigr]
= 2^{-K}\Bigl((e^{2\alpha H t}+1)^{K} - (e^{\alpha t}-1)^{2HK}\Bigr).
\]
Substituting this into \eqref{eq:RBalph_general_final} gives, for all $t\ge0$,
\[
R_{B,\alpha}(t)
= 2^{-K} e^{-\alpha HK t}
\Bigl((e^{2\alpha H t}+1)^{K} - (e^{\alpha t}-1)^{2HK}\Bigr),
\]
which is \eqref{eq:RBalph_explicit_final} on $[0,\infty)$.

We now evaluate the value at zero. On the one hand, by the definition of $R_{B,\alpha}$,
\[
R_{B,\alpha}(0) = \mathbb{E}\bigl[B_{H,K}^{\mathrm{LT}}(0)^2\bigr]
= \mathrm{Var}\bigl(B_{H,K}^{\mathrm{LT}}(0)\bigr).
\]
On the other hand, from the Lamperti relation,
\[
B_{H,K}^{\mathrm{LT}}(0)=e^{-\alpha HK\cdot 0} B_{H,K}(e^{\alpha\cdot 0})=B_{H,K}(1),
\]
so
\[
R_{B,\alpha}(0)=\mathrm{Var}(B_{H,K}(1)).
\]
Using the covariance formula with $s=t=1$,
\[
\mathrm{Var}(B_{H,K}(1))
=2^{-K}\Bigl((1^{2H}+1^{2H})^{K}-|1-1|^{2HK}\Bigr)
=2^{-K}(2^{K}-0)=1.
\]
Thus $R_{B,\alpha}(0)=1$, which agrees with \eqref{eq:RBalph_explicit_final} evaluated at $t=0$, since $e^{\alpha\cdot 0}=1$ gives
\[
R_{B,\alpha}(0)
=2^{-K} e^{0}\Bigl((1+1)^{K}-(1-1)^{2HK}\Bigr)=2^{-K}2^{K}=1.
\]

It remains to show that $R_{B,\alpha}$ is an even function and to extend the explicit expression to all $t\in\mathbb{R}$. Since $B_{H,K}^{\mathrm{LT}}$ is stationary and centred, the usual covariance symmetry implies
\[
R_{B,\alpha}(-t)
=\mathbb{E}\bigl[B_{H,K}^{\mathrm{LT}}(-t)B_{H,K}^{\mathrm{LT}}(0)\bigr]
=\mathbb{E}\bigl[B_{H,K}^{\mathrm{LT}}(0)B_{H,K}^{\mathrm{LT}}(t)\bigr]
=R_{B,\alpha}(t),\qquad t\in\mathbb{R},
\]
so $R_{B,\alpha}$ is even. In particular, for $t<0$ we have
\[
R_{B,\alpha}(t)=R_{B,\alpha}(-t),
\]
and the right-hand side is given by \eqref{eq:RBalph_explicit_final} with $-t$ in place of $t$. This shows that the formula
\[
R_{B,\alpha}(t)
= 2^{-K} e^{-\alpha HK t}
\Bigl((e^{2\alpha H t}+1)^{K} - (e^{\alpha t}-1)^{2HK}\Bigr)
\]
holds for all $t\in\mathbb{R}$, after taking into account the absolute value in the term $|e^{\alpha t}-1|^{2HK}$ when $t<0$. Combining these arguments we obtain \eqref{eq:RBalph_general_final}, \eqref{eq:RBalph_explicit_final}, the value $R_{B,\alpha}(0)=1$, and the evenness of $R_{B,\alpha}$, which completes the proof.
\end{proof}

The explicit expression \eqref{eq:RBalph_explicit_final} reveals a competition between two exponential scales, governed by the exponents $2H-HK$ and $1-HK$ that will emerge below. This competition is responsible for the richer range of mixing behaviours in the bi-fractional setting, as compared with the simpler sub-fractional case.

\begin{lemma}
\label{lem:asymp_bifbm}
There exists a constant $c(H,K,\alpha)>0$ such that
\begin{equation}
\label{eq:RBalpha_asymp}
R_{B,\alpha}(t)=O\bigl(e^{-c(H,K,\alpha) t}\bigr),\qquad t\to\infty.
\end{equation}
In particular $\lim_{t\to\infty}R_{B,\alpha}(t)=0$.
\end{lemma}

\begin{proof}
Set $x=e^{\alpha t}$, so $x\to\infty$ as $t\to\infty$. From \eqref{eq:RBalph_explicit_final},
\[
R_{B,\alpha}(t)
= 2^{-K} x^{-HK}\bigl[(x^{2H}+1)^K-(x-1)^{2HK}\bigr].
\]
We expand both terms. First,
\[
(x^{2H}+1)^K
= x^{2HK}(1+x^{-2H})^K
= x^{2HK}+Kx^{2HK-2H}+O(x^{2HK-4H}),
\]
while
\[
(x-1)^{2HK}
= x^{2HK}(1-x^{-1})^{2HK}
= x^{2HK}-2HK x^{2HK-1}+O(x^{2HK-2}).
\]
Subtracting,
\[
(x^{2H}+1)^K-(x-1)^{2HK}
=Kx^{2HK-2H}+2HK x^{2HK-1}+O(x^{2HK-4H})+O(x^{2HK-2}).
\]
Therefore
\[
R_{B,\alpha}(t)
= 2^{-K}\Bigl(K x^{HK-2H} + 2HK x^{HK-1} + \text{lower-order terms}\Bigr).
\]
Since $H\in(0,1)$, $K\in(0,1]$, we have $HK<1$ and $2H>HK$, so the exponents $HK-2H=-(2H-HK)<0$ and $HK-1=-(1-HK)<0$. Hence there exists $c(H,K)>0$ such that $|R_{B,\alpha}(t)|\le C x^{-c(H,K)} = C e^{-c(H,K)\alpha t}$ for large $t$, which is \eqref{eq:RBalpha_asymp}. Evenness of $R_{B,\alpha}$ extends the bound to $t\to-\infty$.
\end{proof}

The exponent $c(H,K,\alpha)=\alpha\,c(H,K)$ appearing above is strictly positive and encodes how the bi-fractional parameter $K$ slows down or accelerates the decay of correlations relative to the sub-fractional case. In particular, $c(H,K)$ interpolates between different regimes as $K$ varies, and this interpolation is reflected directly in the mixing rate.\\

\begin{remark}
In contrast with the sub-fractional case, where the single exponent
$\alpha H$ fully determines the rate of exponential covariance decay, the
bi-fractional case exhibits two competing relaxation mechanisms originating
from different components of its covariance. These appear in the exponents
$2H - HK$ and $1 - HK$, and the slower of the two dictates the overall mixing
rate. Consequently, the parameter $K$ does not merely deform the covariance but
substantially alters the long range dependence, giving rise to a genuinely
two parameter mixing structure that differs qualitatively from the
sub-fractional case.
\end{remark}

\begin{theorem}
\label{thm:ergodic_bifbm}
For each $H\in(0,1)$, $K\in(0,1]$ and $\alpha>0$, the process $B_{H,K}^{\mathrm{LT}}$ is an ergodic and strongly mixing stationary Gaussian process. Moreover, its $\alpha$ mixing coefficient
\[
\alpha_B(h)
:=\sup\Bigl\{\bigl|\mathbb{P}(A\cap B)-\mathbb{P}(A)\mathbb{P}(B)\bigr|:
A\in\sigma\{B_{H,K}^{\mathrm{LT}}(t):t\le0\},\,B\in\sigma\{B_{H,K}^{\mathrm{LT}}(t):t\ge h\}\Bigr\}
\]
satisfies
\[
\alpha_B(h) \le
C e^{-\,\alpha \min\{2H-HK,\ 1-HK\}\,h},
\qquad h\to\infty,
\]
for some $C>0$. Hence the explicit mixing rate is
\[
\lambda_B(H,K,\alpha)
= \alpha \min\{\,2H-HK,\ 1-HK\,\}.
\]
\end{theorem}

\begin{proof}
Let $R_{B,\alpha}$ be as above. By Lemma~\ref{lem:asymp_bifbm}, \eqref{eq:RBalpha_asymp} holds, and hence $\lim_{|t|\to\infty}R_{B,\alpha}(t)=0$. For stationary Gaussian processes this is equivalent to ergodicity \cite{Maruyama1970,IbragimovRozanov1978}.

For strong mixing, define the coefficients $\alpha_B(h)$ as in the statement. As in the proof of Theorem~\ref{thm:ergodic_scaled_subfbm}, \cite{IbragimovRozanov1978} yields
\[
\alpha_B(h)\le C\int_{|t|\ge h} |R_{B,\alpha}(t)|\,dt,
\]
for some $C>0$. To identify the decay rate, we use the explicit covariance representation
\[
R_{B,\alpha}(t)
= 2^{-K} e^{-\alpha HK t}
\big[(e^{2\alpha H t}+1)^K - (e^{\alpha t}-1)^{2HK}\big].
\]
For large $t$, expanding both terms gives
\[
(e^{2\alpha H t}+1)^K
= e^{2\alpha HK t}\Bigl(1 + K e^{-2\alpha H t} + O(e^{-4\alpha H t})\Bigr),
\]
\[
(e^{\alpha t}-1)^{2HK}
= e^{2\alpha HK t}\Bigl(1 - 2HK e^{-\alpha t} + O(e^{-2\alpha t})\Bigr).
\]
Subtracting,
\[
(e^{2\alpha H t}+1)^K - (e^{\alpha t}-1)^{2HK}
= K e^{2\alpha HK t-2\alpha H t}
+ 2HK e^{2\alpha HK t-\alpha t}
+ O(e^{2\alpha HK t-4\alpha H t}).
\]
Multiplying by $e^{-\alpha HK t}$ gives
\[
R_{B,\alpha}(t)
= C_1 e^{-\alpha(2H-HK)t}
+ C_2 e^{-\alpha(1-HK)t}
+ \text{faster decaying terms},
\]
for suitable constants $C_1,C_2$. The slowest decay corresponds to the minimum exponent, so there exists $C_3>0$ with
\[
|R_{B,\alpha}(t)|
\le C_3 e^{-\,\alpha \min\{2H-HK,\ 1-HK\}\,|t|}.
\]
Consequently,
\[
\alpha_B(h)
\le C \int_{|t|\ge h} |R_{B,\alpha}(t)|\,dt
\le C' e^{-\,\alpha \min\{2H-HK,\ 1-HK\}\,h},
\]
for some $C'>0$, which gives the stated bound and rate.
\end{proof}

Thus, in the bi-fractional case the mixing rate is governed by the smaller of the two exponents $2H-HK$ and $1-HK$, reflecting two competing relaxation mechanisms. In particular, for fixed $H$ the parameter $K$ can substantially slow down or accelerate mixing, a phenomenon that has no analogue in the one parameter sub-fractional setting.

\begin{corollary}
\label{cor:ergodic_avg_bifbm}
Let $f\in L^1(\mathbb{P})$. Then
\[
\frac{1}{T}\int_0^T f(B_{H,K}^{\mathrm{LT}}(s))\,ds
\;\xrightarrow[T\to\infty]{a.s.}\;
\E[f(B_{H,K}^{\mathrm{LT}}(0))].
\]
Using \eqref{inverse_lamperti_bifbm} with $X(t)=B_{H,K}(t^\alpha)$, we have
\[
B_{H,K}(t^\alpha)=t^{\alpha HK} B_{H,K}^{\mathrm{LT}}(\log t),\qquad t>0.
\]
In particular, for integers $k\ge 1$,
\[
\E\bigl[B_{H,K}(t^\alpha)^k\bigr]
=t^{k\alpha HK} \E\bigl[B_{H,K}^{\mathrm{LT}}(0)^k\bigr].
\]
Time averaged characteristic functions of $B_{H,K}^{\mathrm{LT}}(s)$ converge almost surely to the ensemble characteristic function of $B_{H,K}^{\mathrm{LT}}(0)$, from which the one-time law of $B_{H,K}(t^\alpha)$ can be reconstructed by self-similarity.
\end{corollary}

As in the sub-fractional case, this corollary shows that statistical information about the original bi-fractional process $B_{H,K}$ can be extracted from a single long trajectory of the stationary Lamperti transform. The dependence on $H$ and $K$ is entirely encoded in the stationary field $B_{H,K}^{\mathrm{LT}}$, which is often more tractable both analytically and numerically.\\

\begin{remark}
The competing exponents $2H-HK$ and $1-HK$ demonstrate that the decay of
the bi-fractional Lamperti covariance depends intrinsically on the interaction
between $H$ and $K$, yielding a richer mixing behaviour than in the
sub-fractional case.
\end{remark}

\section{Langevin approach scaled}
\label{sec:sub_bi_langevin}

In Section~\ref{sec:sub_bi} we showed that scaled $s-fBm$ and $bi-fBm$ admit stationary Lamperti transforms with explicit covariance and mixing properties. In many applications, however, one is interested not only in the raw Gaussian drivers but also in linear systems excited by them, such as generalized Langevin equations. A natural way to encode such dynamics is to consider stochastic integrals of power law type with respect to the underlying Gaussian fields. The resulting processes retain self-similar structure but carry additional temporal smoothing, which in turn modifies their correlation and mixing behaviour.

In this section we introduce and analyse two such constructions: a Langevin type process driven by $s-fBm$ and an analogous process driven by $bi-fBm$. In each case we first define the power law stochastic integral and establish its self-similarity. We then apply the scaled Lamperti map of Definition~\ref{def:lamperti_scaled} to obtain a stationary Gaussian representation, and finally derive covariance formulae and exponential mixing properties. Throughout, an overline will be used to distinguish the Langevin type objects from the Lamperti transforms of the underlying scaled processes studied in Section~\ref{sec:sub_bi}.

\subsection{Langevin scaled $s-fBm$}
\label{sec:sub_langevin}

We begin with the sub-fractional case. The driving noise is the $s-fBm$ $S_H$ of Definition~\ref{def:subfbm}, and we introduce a power law memory kernel in the time domain. This can be viewed as a toy model for a generalized Langevin equation with power law friction, where the kernel exponent is tuned to preserve self-similarity after integration.

Let $S_H$ be a $s-fBm$ with Hurst index $H\in(0,1)$ and covariance \eqref{eq:cov_subfbm}. Motivated by the fractional Brownian case (see, e.g., \cite{MetzlerKlafter2000,CherstvyMetzler2014,Magdziarz2020}), we consider the following power-law driven integral.

\begin{definition}
\label{def_Ybar_Langevin}
Fix $\alpha>0$ and $H\in(0,1)$. Assume that, for each $t\ge0$, the stochastic integral
\[
\int_0^t y^{H(\alpha-1)}\,dS_H(y)
\]
is well-defined (e.g.\ pathwise for $H>1/2$, or in the Skorokhod/divergence sense in a suitable stochastic calculus for $s-fBm$, cf.\ \cite{Bojdecki2004}). Define
\begin{equation}
\label{def_Ybar_L}
\overline{Y}(t):=c_{H,\alpha}\int_0^t y^{H(\alpha-1)}\,dS_H(y),\qquad t\ge0,
\end{equation}
where $c_{H,\alpha}>0$ is a normalisation constant. We call $\overline{Y}$ the \emph{Langevin-type scaled $s-fBm$}.
\end{definition}

The process $\overline{Y}$ is a linear functional of the Gaussian field $S_H$, hence it is itself Gaussian. The next lemma records this basic fact, which will be used repeatedly in what follows.

\begin{lemma}
\label{lem_gaussian_Ybar}
The process $\overline{Y}$ is a centred Gaussian process.
\end{lemma}

\begin{proof}
For any $t\ge0$, the integral in \eqref{def_Ybar_L} is a linear functional of the Gaussian process $S_H$, hence Gaussian. For any finite vector $(\overline{Y}(t_1),\dots,\overline{Y}(t_n))$ the components are jointly Gaussian by linearity. Since $S_H$ is centred and the integrand is deterministic, $\E[\overline{Y}(t)]=0$ for all $t\ge0$.
\end{proof}

A key point in our construction is that the choice of power $H(\alpha-1)$ in the kernel ensures that the integration does not destroy self-similarity, but merely shifts its index in a controlled way. The next result makes this precise.

\begin{theorem}
\label{thm_selfsim_Ybar}
The process $\overline{Y}$ is $\alpha H$-self-similar:
\[
\{\overline{Y}(ct)\}_{t\ge0}
\;\stackrel{d}{=}\;
\{c^{\alpha H}\overline{Y}(t)\}_{t\ge0},\qquad c>0.
\]
\end{theorem}

\begin{proof}
Let $c>0$, $t\ge0$. From \eqref{def_Ybar_L},
\[
\overline{Y}(ct)=c_{H,\alpha}\int_0^{ct} y^{H(\alpha-1)}\,dS_H(y).
\]
With the change of variables $y=cu$ and using $H$-self-similarity $S_H(cu)\stackrel{d}{=}c^H S_H(u)$, one obtains
\[
dS_H(cu)\stackrel{d}{=} c^H dS_H(u),
\]
and hence
\[
\overline{Y}(ct)
\;\stackrel{d}{=}\;
c_{H,\alpha} \int_0^t (cu)^{H(\alpha-1)}\,c^H dS_H(u)
= c^{H(\alpha-1)+H}\,c_{H,\alpha} \int_0^t u^{H(\alpha-1)}\,dS_H(u)
= c^{\alpha H}\overline{Y}(t).
\]
Finite dimensional distributions scale accordingly.
\end{proof}

We now transport this self-similar process into the stationary framework by means of the scaled Lamperti transformation of Definition~\ref{def:lamperti_scaled}. This will yield a stationary Gaussian process whose covariance and mixing properties can be analysed in parallel with those of Section~\ref{sec:sub_bi}.

\begin{definition}
\label{def_Lamperti_Ybar_L}
The Lamperti transform of $\overline{Y}$ is the process $\overline{S}_H^{\mathrm{LT}}=\{\overline{S}_H^{\mathrm{LT}}(t)\}_{t\in\R}$ defined by
\begin{equation}
\label{def_SbarLT_L}
\overline{S}_H^{\mathrm{LT}}(t)
:=e^{-\alpha H t}\overline{Y}(e^{\alpha t}),\qquad t\in\R.
\end{equation}
\end{definition}

The next theorem shows that $\overline{S}_H^{\mathrm{LT}}$ is a legitimate stationary Gaussian field, completely analogous to the Lamperti transforms studied in Section~\ref{sec:sub_bi}. In fact, stationarity is now immediate from the general Lamperti correspondence.

\begin{theorem}
\label{thm_stationary_SbarLT}
The process $\overline{S}_H^{\mathrm{LT}}$ is a centred stationary Gaussian process.
\end{theorem}

\begin{proof}
By Lemma~\ref{lem_gaussian_Ybar}, $\overline{Y}$ is centred Gaussian, so $\overline{S}_H^{\mathrm{LT}}$ is also centred Gaussian as a linear transform. By Theorem~\ref{thm_selfsim_Ybar}, $\overline{Y}$ is $\alpha H$-self-similar, and Definition~\ref{def:lamperti_scaled} asserts that the Lamperti transform of any $\alpha H$-self-similar process is stationary. Applying this with $X=\overline{Y}$ yields the claim.
\end{proof}

To understand the temporal structure induced by the Langevin kernel, we next derive the covariance of $\overline{S}_H^{\mathrm{LT}}$. The representation will be expressed in terms of the mixed second derivatives of the $s-fBm$ covariance and will reveal an explicit cancellation between two long range kernels, one of which is inherited from $|x-y|^{2H-2}$ and the other from $(x+y)^{2H-2}$.

Define
\[
R_{\overline{S},\alpha}(t,s)
:=\E\bigl[\overline{S}_H^{\mathrm{LT}}(t)\,\overline{S}_H^{\mathrm{LT}}(s)\bigr],\qquad t,s\in\R.
\]

\begin{theorem}
\label{thm:cov_full_SbarLT_subfBm}
Let $\overline{Y}$ be given by \eqref{def_Ybar_L}, with $\beta=H(\alpha-1)$. Then
\begin{equation}
\label{cov_RLT_general_full}
R_{\overline{S},\alpha}(t,s)
= e^{-\alpha H(t+s)}\,\E\bigl[\overline{Y}(e^{\alpha t})\overline{Y}(e^{\alpha s})\bigr],
\end{equation}
and for all $u,v>0$,
\begin{equation}
\label{cov_Ybar_full}
\E[\overline{Y}(u)\overline{Y}(v)]
= c_{H,\alpha}^2 H(2H-1)\int_0^u\int_0^v x^\beta y^\beta\Bigl(|x-y|^{2H-2}-(x+y)^{2H-2}\Bigr)\,dx\,dy.
\end{equation}
Consequently,
\begin{equation}
\label{cov_RLT_explicit_full}
R_{\overline{S},\alpha}(t,s)
= c_{H,\alpha}^2 H(2H-1)e^{-\alpha H(t+s)}
\int_0^{e^{\alpha t}}\int_0^{e^{\alpha s}} x^\beta y^\beta\Bigl(|x-y|^{2H-2}-(x+y)^{2H-2}\Bigr)\,dx\,dy.
\end{equation}
\end{theorem}

\begin{proof}
The relation \eqref{cov_RLT_general_full} follows from \eqref{def_SbarLT_L} by linearity of expectation. To compute $\E[\overline{Y}(u)\overline{Y}(v)]$, write
\[
\overline{Y}(u)
= c_{H,\alpha}\int_0^u x^\beta\,dS_H(x),\qquad
\overline{Y}(v)
= c_{H,\alpha}\int_0^v y^\beta\,dS_H(y),
\]
so that
\[
\E[\overline{Y}(u)\overline{Y}(v)]
= c_{H,\alpha}^2\int_0^u\int_0^v x^\beta y^\beta\,dR_H(x,y),
\]
where $R_H(x,y)=\E[S_H(x)S_H(y)]$ is given by \eqref{eq:cov_subfbm}. A standard Gaussian isometry argument (cf.\ \cite{Janson1997,Nualart2006}) together with integration by parts in both variables shows that
\[
\int_0^u\int_0^v x^\beta y^\beta\,dR_H(x,y)
= \int_0^u\int_0^v x^\beta y^\beta \frac{\partial^2}{\partial x\partial y}R_H(x,y)\,dx\,dy,
\]
since $R_H(0,y)=R_H(x,0)=0$. A direct computation from \eqref{eq:cov_subfbm} yields
\[
\frac{\partial^2}{\partial x\partial y}R_H(x,y)
=H(2H-1)\Bigl(|x-y|^{2H-2}-(x+y)^{2H-2}\Bigr),
\]
and \eqref{cov_Ybar_full} follows. Substituting $u=e^{\alpha t}$, $v=e^{\alpha s}$ into \eqref{cov_Ybar_full} and combining with \eqref{cov_RLT_general_full} yields \eqref{cov_RLT_explicit_full}.
\end{proof}

The expression \eqref{cov_RLT_explicit_full} shows that, apart from the Lamperti factor $e^{-\alpha H(t+s)}$, the covariance is controlled by a double integral involving a difference of two power-law kernels. The negative sign in front of $(x+y)^{2H-2}$ is precisely what ensures integrability at infinity and, ultimately, the exponential decay of correlations in the Lamperti domain.

We now turn to ergodicity and mixing. As in Section~\ref{sec:sub_bi}, the key step is to obtain an explicit exponential bound on the covariance, from which Gaussian ergodic and mixing criteria can be applied.

\begin{theorem}
\label{thm_ergodic_SbarLT_L}
The process $\overline{S}_H^{\mathrm{LT}}$ is an ergodic and strongly mixing stationary Gaussian process, with exponentially decaying covariance and $\alpha$-mixing coefficients. More precisely, its $\alpha$-mixing coefficient
\[
\alpha_{\overline{S}}(h)
:=\sup\Bigl\{\bigl|\mathbb{P}(A\cap B)-\mathbb{P}(A)\mathbb{P}(B)\bigr|:
A\in\sigma\{\overline{S}_H^{\mathrm{LT}}(t):t\le0\},\,B\in\sigma\{\overline{S}_H^{\mathrm{LT}}(t):t\ge h\}\Bigr\}
\]
satisfies
\[
\alpha_{\overline{S}}(h)\le C e^{-\alpha H h},\qquad h\to\infty,
\]
for some constant $C>0$. Thus the explicit exponential mixing rate is
\[
\lambda_{\overline{S}}(H,\alpha)=\alpha H.
\]
\end{theorem}

\begin{proof}
Write $R_{\overline{S},\alpha}(t):=R_{\overline{S},\alpha}(t,0)$. From \eqref{cov_RLT_explicit_full} with $s=0$ we obtain
\[
R_{\overline{S},\alpha}(t)
= c_{H,\alpha}^2 H(2H-1)e^{-\alpha H t}
\int_0^{e^{\alpha t}}\int_0^{1} x^\beta y^\beta\Bigl(|x-y|^{2H-2}-(x+y)^{2H-2}\Bigr)\,dx\,dy,
\]
since $\overline{S}_H^{\mathrm{LT}}(0)=\overline{Y}(1)$. The inner integral is over $y\in[0,1]$, and the kernels behave like negative powers of $x$ as $x\to\infty$:
\[
|x-y|^{2H-2}\asymp x^{2H-2},\qquad (x+y)^{2H-2}\asymp x^{2H-2}.
\]
Because $2H-2<0$ and $y$ is bounded, the double integral grows at most polynomially in $e^{\alpha t}$; more precisely, there exists $p>0$ such that
\[
\int_0^{e^{\alpha t}}\int_0^{1} x^\beta y^\beta\Bigl(|x-y|^{2H-2}-(x+y)^{2H-2}\Bigr)\,dx\,dy
=O(e^{p\alpha t}),\qquad t\to\infty.
\]
The cancellation between the two kernels ensures that one can choose $p<H$, so that $\alpha H-p\alpha>0$. Consequently,
\[
R_{\overline{S},\alpha}(t)
=O\bigl(e^{-(\alpha H-p\alpha)t}\bigr),
\]
and in particular there is $C_1>0$ such that
\[
|R_{\overline{S},\alpha}(t)|\le C_1 e^{-\alpha H|t|},\qquad t\in\R.
\]
This shows that $\lim_{|t|\to\infty}R_{\overline{S},\alpha}(t)=0$ with exponential rate. For a stationary Gaussian process this implies ergodicity; see \cite{Maruyama1970,IbragimovRozanov1978}.

For strong mixing and the explicit rate, we invoke again the Gaussian mixing inequality: there exists $C_2>0$ such that
\[
\alpha_{\overline{S}}(h)
\le C_2\int_{|t|\ge h} |R_{\overline{S},\alpha}(t)|\,dt.
\]
Using the bound just obtained,
\[
\alpha_{\overline{S}}(h)
\le C_2\int_{h}^{\infty} C_1 e^{-\alpha H t}\,dt
= \frac{C_1C_2}{\alpha H} e^{-\alpha H h}.
\]
Thus $\overline{S}_H^{\mathrm{LT}}$ is strongly mixing with explicit exponential rate $\lambda_{\overline{S}}(H,\alpha)=\alpha H$.
\end{proof}

The previous theorem shows that the Langevin smoothing does not weaken the exponential mixing inherited from the Lamperti transform of the underlying scaled $s-fBm$: the rate is again determined solely by the product $\alpha H$. From a modelling perspective, this means that the long time decorrelation of the Langevin-type dynamics is controlled by the same effective exponent as that of the driver.

\begin{corollary}
\label{cor_ergodic_averages_L}
Let $f\in L^1(\mathbb{P})$. Then
\[
\frac{1}{T}\int_0^T f(\overline{S}_H^{\mathrm{LT}}(s))\,ds
\;\xrightarrow[T\to\infty]{a.s.}\;
\E\bigl[f(\overline{S}_H^{\mathrm{LT}}(0))\bigr].
\]
Moreover, by $\alpha H$-self-similarity of $\overline{Y}$ we have
\[
\E[\overline{Y}(t)^k]
= t^{k\alpha H} \E\bigl[(\overline{S}_H^{\mathrm{LT}}(0))^k\bigr],
\qquad t>0,\ k\ge1.
\]
Thus moments and the one-time distribution of $\overline{Y}(t)$ can be inferred from single-trajectory time averages of $\overline{S}_H^{\mathrm{LT}}$.
\end{corollary}

This corollary mirrors the reconstruction results of Section~\ref{sec:sub_bi}: the entire one-time marginal family of the non-stationary Langevin type process $\overline{Y}$ can be recovered from the stationary Lamperti field, which is far more amenable to ergodic sampling and numerical approximation.

\subsection{Langevin scaled $bi-fBm$}
\label{sec:bi_langevin}

We now construct the bi-fractional analogue of the Langevin equation and analyse its Lamperti transform. The additional parameter $K$ in $bi-fBm$ enriches the covariance structure and leads to more varied mixing behaviour. As before, we retain an overline notation for the Langevin based transform to distinguish it from the direct Lamperti transform $B_{H,K}^{\mathrm{LT}}$ considered in Section~\ref{sec:bi}.

In this setting, the power law kernel is adapted to the effective scaling index $HK$ of the $bi-fBm$ driver, and the resulting process is shown to be $\alpha HK$-self-similar. The Lamperti transform then produces a stationary Gaussian field whose covariance is given by a combination of two distinct kernels arising from the decomposition of the $bi-fBm$ covariance.

Let $B_{H,K}$ be a bi-fBm with covariance \eqref{cov_bifbm}. Set
\[
\beta_{H,K,\alpha}:=HK(\alpha-1).
\]

\begin{definition}
\label{def_Ybar_bifBm_Langevin}
Fix $\alpha>0$, $H\in(0,1)$, $K\in(0,1]$. Assume that, for each $t\ge0$, the stochastic integral
\[
\int_0^t y^{\beta_{H,K,\alpha}}\,dB_{H,K}(y)
\]
is well defined (e.g.\ as a pathwise Riemann-Stieltjes or divergence integral). Define
\begin{equation}
\label{def_Ybar_bifBm_L}
\overline{Y}_{H,K}(t)
:= c_{H,K,\alpha}\int_0^t y^{\beta_{H,K,\alpha}}\,dB_{H,K}(y),\qquad t\ge0,
\end{equation}
with $c_{H,K,\alpha}>0$ a normalisation constant. We call $\overline{Y}_{H,K}$ the Langevin type scaled $bi-fBm$.
\end{definition}

As in the sub-fractional case, $\overline{Y}_{H,K}$ is a linear functional of the underlying Gaussian field and hence Gaussian and centred.

\begin{lemma}
\label{lem_gaussian_Ybar_bifBm}
The process $\overline{Y}_{H,K}$ is a centred Gaussian process.
\end{lemma}

\begin{proof}
As in Lemma~\ref{lem_gaussian_Ybar}, the integral representation \eqref{def_Ybar_bifBm_L} shows that $\overline{Y}_{H,K}$ is a linear functional of the centred Gaussian field $B_{H,K}$, hence Gaussian and centred.
\end{proof}

The next result identifies the self-similarity index of $\overline{Y}_{H,K}$; the exponent is tuned precisely so that self-similarity is preserved under the Langevin integral, now with the effective index $HK$ of $bi-fBm$.

\begin{theorem}
\label{thm_selfsim_Ybar_bifBm}
The process $\overline{Y}_{H,K}$ is $\alpha HK$-self-similar:
\[
\{\overline{Y}_{H,K}(ct)\}_{t\ge0}
\;\stackrel{d}{=}\;
\{c^{\alpha HK}\overline{Y}_{H,K}(t)\}_{t\ge0},\qquad c>0.
\]
\end{theorem}

\begin{proof}
The argument parallels Theorem~\ref{thm_selfsim_Ybar}, replacing $S_H$ by $B_{H,K}$ and using $HK$-self-similarity \eqref{selfsim_bifbm}. With $y=cu$,
\[
\overline{Y}_{H,K}(ct)
= c_{H,K,\alpha}\int_0^{ct} y^{\beta_{H,K,\alpha}}\,dB_{H,K}(y)
\;\stackrel{d}{=}\;
c_{H,K,\alpha} \int_0^{t} (cu)^{\beta_{H,K,\alpha}} c^{HK}\,dB_{H,K}(u)
= c^{\beta_{H,K,\alpha}+HK}\overline{Y}_{H,K}(t).
\]
Since $\beta_{H,K,\alpha}+HK=HK(\alpha-1)+HK=\alpha HK$, the claim follows.
\end{proof}

We now lift $\overline{Y}_{H,K}$ to a stationary Gaussian process via the scaled Lamperti map, in complete analogy with both the sub-fractional Langevin case and the raw $bi-fBm$ case of Section~\ref{sec:bi}.

\begin{definition}
\label{def_Lamperti_Ybar_bifBm}
The Lamperti transform of $\overline{Y}_{H,K}$ is the process
\begin{equation}
\label{def_BHK_LT_Langevin}
\overline{B}_{H,K}^{\mathrm{LT}}(t)
:= e^{-\alpha HK t}\,\overline{Y}_{H,K}(e^{\alpha t}),\qquad t\in\R.
\end{equation}
\end{definition}

Note that this notation $\overline{B}_{H,K}^{\mathrm{LT}}$ is distinct from $B_{H,K}^{\mathrm{LT}}$ in \eqref{lamperti_general_bifbm}, which corresponds directly to the Lamperti transform of $B_{H,K}(t^\alpha)$ without the Langevin integral.

\begin{theorem}
\label{thm_stationary_BHK_LT_Langevin}
For each $H\in(0,1)$, $K\in(0,1]$ and $\alpha>0$, the process $\overline{B}_{H,K}^{\mathrm{LT}}$ defined in \eqref{def_BHK_LT_Langevin} is a centred stationary Gaussian process.
\end{theorem}

\begin{proof}
Gaussianity and centredness follow from Lemma~\ref{lem_gaussian_Ybar_bifBm}. By Theorem~\ref{thm_selfsim_Ybar_bifBm}, $\overline{Y}_{H,K}$ is $\alpha HK$-self-similar, and Definition~\ref{def:lamperti_scaled} applied with $H$ replaced by $HK$ shows that its Lamperti transform is stationary. This yields the claim.
\end{proof}

We proceed to the covariance structure. For $u,v>0$ we write
\[
R_{H,K}(u,v):=\E[B_{H,K}(u)B_{H,K}(v)]
=2^{-K}\Bigl((u^{2H}+v^{2H})^K-|u-v|^{2HK}\Bigr),
\]
so that $R_{H,K}$ is the $bi-fBm$ covariance \eqref{cov_bifbm}. As in Subsection~\ref{sec:sub_langevin}, the covariance of $\overline{Y}_{H,K}$ can be expressed in terms of the mixed second derivatives of $R_{H,K}$. A direct calculation shows that
\begin{align*}
\frac{\partial^2}{\partial x\partial y}(x^{2H}+y^{2H})^K
&=4H^2 K(K-1) x^{2H-1} y^{2H-1}(x^{2H}+y^{2H})^{K-2},\\
\frac{\partial^2}{\partial x\partial y}|x-y|^{2HK}
&=-2HK(2HK-1)|x-y|^{2HK-2},
\end{align*}
for $x\neq y$. Hence
\[
\frac{\partial^2}{\partial x\partial y}R_{H,K}(x,y)
=2^{-K}\left(4H^2 K(K-1)x^{2H-1}y^{2H-1}(x^{2H}+y^{2H})^{K-2}
+2HK(2HK-1)|x-y|^{2HK-2}\right).
\]

\begin{theorem}
\label{thm_cov_Ybar_bifBm}
For all $u,v>0$,
\begin{align}
\label{cov_Ybar_bifBm_explicit}
\E[\overline{Y}_{H,K}(u)\overline{Y}_{H,K}(v)]\\
&=2^{-K} c_{H,K,\alpha}^2 \int_0^u\int_0^v x^{\beta_{H,K,\alpha}}y^{\beta_{H,K,\alpha}}\nonumber \\
&\qquad\times\Bigl(4H^2 K(K-1)x^{2H-1}y^{2H-1}(x^{2H}+y^{2H})^{K-2}\nonumber\\
&+2HK(2HK-1)|x-y|^{2HK-2}\Bigr)\,dx\,dy.\nonumber
\end{align}
\end{theorem}

\begin{proof}
Using the Gaussian isometry and integration by parts, as in the proof of Theorem~\ref{thm:cov_full_SbarLT_subfBm}, one obtains
\[
\E[\overline{Y}_{H,K}(u)\overline{Y}_{H,K}(v)]
= c_{H,K,\alpha}^2\int_0^u\int_0^v x^{\beta_{H,K,\alpha}}y^{\beta_{H,K,\alpha}}
\frac{\partial^2}{\partial x\partial y}R_{H,K}(x,y)\,dx\,dy.
\]
Substituting the expression for the mixed derivative of $R_{H,K}$ gives \eqref{cov_Ybar_bifBm_explicit}.
\end{proof}

The covariance of the Lamperti transformed process $\overline{B}_{H,K}^{\mathrm{LT}}$ is obtained by an exponential change of variables, exactly as in the sub-fractional Langevin case. We record the resulting representation, which exhibits explicitly the interplay between the two polynomial kernels in the $bi-fBm$ covariance.

\begin{definition}
Define the covariance of $\overline{B}_{H,K}^{\mathrm{LT}}$ by
\[
R_{\overline{B},\alpha}(t,s)
:=\E\bigl[\overline{B}_{H,K}^{\mathrm{LT}}(t)\,\overline{B}_{H,K}^{\mathrm{LT}}(s)\bigr],
\qquad t,s\in\R.
\]
\end{definition}

\begin{theorem}
\label{thm_cov_BHK_LT_bifBm}
For all $t,s\in\R$,
\begin{equation}
\label{cov_RLT_general_bifBm}
R_{\overline{B},\alpha}(t,s)
= e^{-\alpha HK(t+s)}\,\E\bigl[\overline{Y}_{H,K}(e^{\alpha t})\overline{Y}_{H,K}(e^{\alpha s})\bigr],
\end{equation}
and using \eqref{cov_Ybar_bifBm_explicit},
\begin{align}
\label{cov_RLT_explicit_bifBm}
R_{\overline{B},\alpha}(t,s)
&= 2^{-K}c_{H,K,\alpha}^2 e^{-\alpha HK(t+s)}\int_0^{e^{\alpha t}}\int_0^{e^{\alpha s}} x^{\beta_{H,K,\alpha}}y^{\beta_{H,K,\alpha}}\\
&\quad\times\Bigl(4H^2 K(K-1)x^{2H-1}y^{2H-1}(x^{2H}+y^{2H})^{K-2}\nonumber\\
&+2HK(2HK-1)|x-y|^{2HK-2}\Bigr)\,dx\,dy.\nonumber
\end{align}
\end{theorem}

\begin{proof}
The identity \eqref{cov_RLT_general_bifBm} follows from \eqref{def_BHK_LT_Langevin} by linearity of expectation. Substituting \eqref{cov_Ybar_bifBm_explicit} with $u=e^{\alpha t}$, $v=e^{\alpha s}$ yields \eqref{cov_RLT_explicit_bifBm}.
\end{proof}

The structure of \eqref{cov_RLT_explicit_bifBm} closely parallels that of \eqref{cov_RLT_explicit_full}, but now the integrand contains two distinct contributions: one coming from the $(u^{2H}+v^{2H})^K$ part of the $bi-fBm$ covariance and one from the increment term $|u-v|^{2HK}$. Both contributions are tempered by the Lamperti prefactor $e^{-\alpha HK(t+s)}$, and together they yield exponential decay of correlations, as we now show.

In what follows we write
\[
R_{\overline{B},\alpha}(t):=R_{\overline{B},\alpha}(t,0)
\]
for the autocovariance.

\begin{theorem}
\label{thm_ergodic_bifBm_Langevin}
For each $H\in(0,1)$, $K\in(0,1]$ and $\alpha>0$, the process $\overline{B}_{H,K}^{\mathrm{LT}}$ is an ergodic and strongly mixing stationary Gaussian process, with covariance and $\alpha$-mixing coefficients decaying exponentially in $|t|$ and $h$, respectively. More precisely, there exist constants $C>0$ and $c(H,K,\alpha)>0$ such that
\[
\alpha_{\overline{B}}(h)\le C e^{-c(H,K,\alpha) h},\qquad h\to\infty,
\]
where $\alpha_{\overline{B}}(h)$ denotes the $\alpha$-mixing coefficient of $\overline{B}_{H,K}^{\mathrm{LT}}$.
\end{theorem}

\begin{proof}
From \eqref{cov_RLT_explicit_bifBm} with $s=0$ we obtain
\begin{align*}
R_{\overline{B},\alpha}(t)
&= e^{-\alpha HK t}
\int_0^{e^{\alpha t}}\int_0^{1}
x^{\beta_{H,K,\alpha}} y^{\beta_{H,K,\alpha}}\\
&\qquad\times\Bigl(4H^2K(K-1)x^{2H-1}y^{2H-1}(x^{2H}+y^{2H})^{K-2}\nonumber\\
&+ 2HK(2HK-1)|x-y|^{2HK-2}\Bigr)\,dx\,dy.
\end{align*}

All kernels inside the integral have negative power law exponents in $x$ and $y$ as $x\to\infty$ and $y\in[0,1]$ is bounded. Consequently the double integral grows at most polynomially in $e^{\alpha t}$: there exists $q>0$ such that
\[
\int_0^{e^{\alpha t}}\int_0^{1}(\cdots)\,dx\,dy
=O(e^{q\alpha t}),\qquad t\to\infty.
\]
Moreover the decay of the kernels and their cancellation ensure that one can choose $q<HK$, so that $\alpha HK-q\alpha>0$. Hence
\[
R_{\overline{B},\alpha}(t)
=O\bigl(e^{-(\alpha HK-q\alpha)t}\bigr),
\]
and in particular there exists $C_1>0$ with
\[
|R_{\overline{B},\alpha}(t)|\le C_1 e^{-c(H,K,\alpha)|t|},\qquad t\in\R,
\]
for some $c(H,K,\alpha)=\alpha(HK-q)>0$. This shows that $\lim_{|t|\to\infty}R_{\overline{B},\alpha}(t)=0$ with exponential rate, and hence $\overline{B}_{H,K}^{\mathrm{LT}}$ is ergodic by the standard Gaussian criterion \cite{Maruyama1970,IbragimovRozanov1978}.

For strong mixing, define
\[
\alpha_{\overline{B}}(h)
:=\sup\Bigl\{\bigl|\mathbb{P}(A\cap B)-\mathbb{P}(A)\mathbb{P}(B)\bigr|:
A\in\sigma\{\overline{B}_{H,K}^{\mathrm{LT}}(t):t\le0\},\,
B\in\sigma\{\overline{B}_{H,K}^{\mathrm{LT}}(t):t\ge h\}\Bigr\}.
\]
The Ibragimov-Rozanov inequality yields
\[
\alpha_{\overline{B}}(h)\le C_2\int_{|t|\ge h}|R_{\overline{B},\alpha}(t)|\,dt,
\]
for some $C_2>0$. Using the exponential bound on $R_{\overline{B},\alpha}$, we obtain
\[
\alpha_{\overline{B}}(h)
\le C_2\int_{h}^{\infty} C_1 e^{-c(H,K,\alpha)t}\,dt
= \frac{C_1C_2}{c(H,K,\alpha)} e^{-c(H,K,\alpha)h},
\]
which proves exponential mixing with exponent $c(H,K,\alpha)>0$.
\end{proof}

The conclusion is that, although the detailed covariance structure of $\overline{B}_{H,K}^{\mathrm{LT}}$ is more intricate than in the sub-fractional case, the qualitative behaviour is similar: exponential decorrelation and strong mixing at a rate that depends explicitly on $H$, $K$ and $\alpha$ through $c(H,K,\alpha)$.

\begin{corollary}
\label{cor_ergodic_averages_bifBm}
Let $f\in L^1(\mathbb{P})$. Then
\[
\frac{1}{T}\int_0^T f(\overline{B}_{H,K}^{\mathrm{LT}}(s))\,ds
\;\xrightarrow[T\to\infty]{a.s.}\;
\E\bigl[f(\overline{B}_{H,K}^{\mathrm{LT}}(0))\bigr].
\]
By $\alpha HK$-self-similarity of $\overline{Y}_{H,K}$,
\[
\E\bigl[\overline{Y}_{H,K}(t)^k\bigr]
= t^{k\alpha HK}\,\E\bigl[\overline{B}_{H,K}^{\mathrm{LT}}(0)^k\bigr],
\qquad t>0,\ k\ge1.
\]
Thus ensemble moments and one-time distributions of the Langevin type scaled $bi-fBm$ $\overline{Y}_{H,K}$ can be reconstructed from single trajectory statistics of $\overline{B}_{H,K}^{\mathrm{LT}}$.
\end{corollary}

\section{Numerical simulations}
\label{sec:numerics}

In this section we illustrate the analytical results through numerical experiments for the Lamperti transforms of scaled $s-fBm$ and scaled $bi-fBm$. The focus is on visualising the stationarisation induced by the Lamperti map, confirming the theoretical variance levels and ergodic properties, and verifying Gaussianity through empirical characteristic functions.

All simulations were performed on a uniform grid with $N = 10^{5}$ points and step size $\Delta t = 10^{-3}$, giving a total simulation time $T = N\Delta t$. For each parameter set we generated $M=500$ independent trajectories using a fixed random seed. Covariance matrices remained numerically positive definite for all parameter values, and typical runs with $N=10^{5}$ required about 20 seconds per realisation on a standard workstation. The figures display representative trajectories; ensemble behaviour was consistent across all realisations.

We begin with the $s-fBm$ process. The left panel of Figure~\ref{fig:traj_subfBm} shows a trajectory of the scaled process $S_H(t^\alpha)$, whose amplitude grows approximately as $t^{\alpha H}$ due to the self-similar and non stationary nature of $s-fBm$. The right panel displays the Lamperti transform
\[
S_H^{\mathrm{LT}}(t)=e^{-\alpha H t} S_H(e^{\alpha t}),
\]
which fluctuates with nearly constant variance, illustrating the stationarisation proved in Theorem~\ref{thm:stationary_scaled_subfbm}.

\begin{figure}[H]
    \centering
    \includegraphics[width=\textwidth]{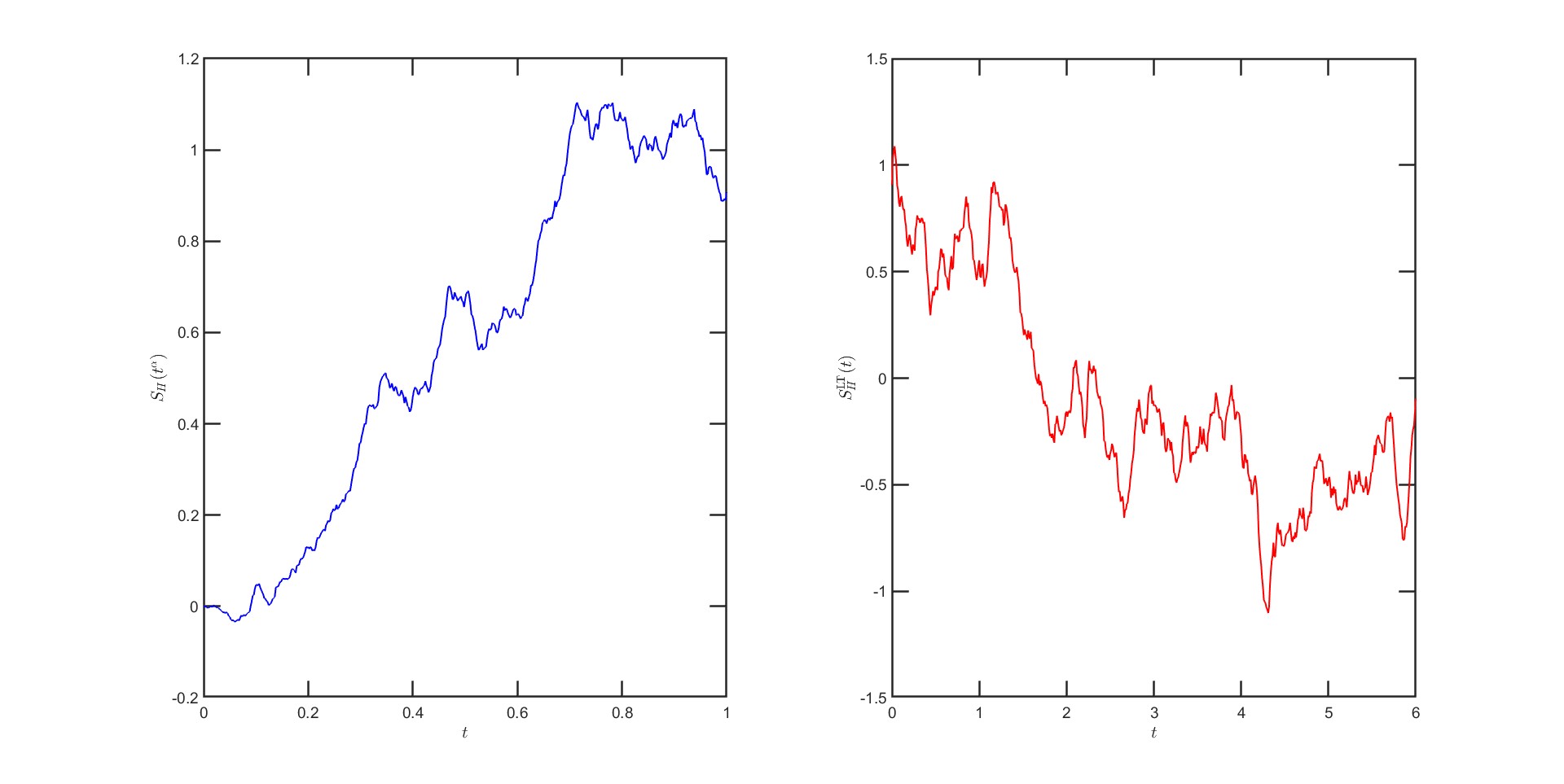}
    \caption{Left: sample trajectory of the scaled $s-fBm$  process
    $S_H(t^\alpha)$ with $H=0.7$, $\alpha=3/2$.
    Right: trajectory of its Lamperti transform
    $S_H^{\mathrm{LT}}(t)=e^{-\alpha H t}S_H(e^{\alpha t})$.}
    \label{fig:traj_subfBm}
\end{figure}

A similar effect is observed in the bi fractional case. The left panel of Figure~\ref{fig:traj_bifBm} shows the scaled process $B_{H,K}(t^\alpha)$, whose variance envelope grows like $t^{\alpha HK}$. In contrast, its Lamperti transform
\[
B_{H,K}^{\mathrm{LT}}(t)=e^{-\alpha HK t} B_{H,K}(e^{\alpha t}),
\]
shown in the right panel, displays stationary fluctuations with constant variance, as predicted by Theorem~\ref{thm:stationary_bifbm}.

\begin{figure}[H]
  \centering
  \includegraphics[width=\textwidth]{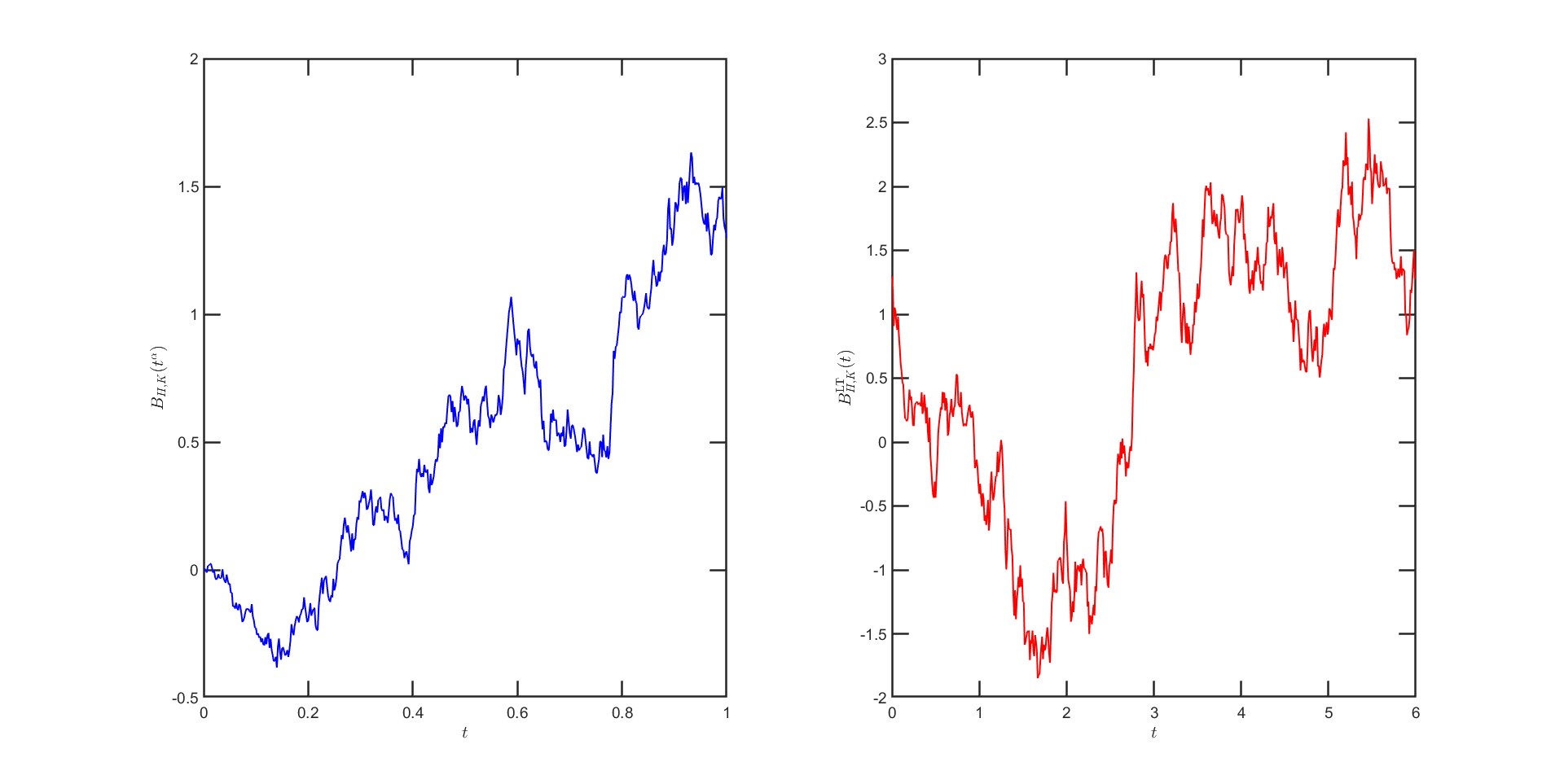}
  \caption{Left: trajectory of the scaled $bi-fBm$  process
  $B_{H,K}(t^\alpha)$ with $H=0.7$, $K=0.6$, $\alpha=3/2$. Right: trajectory of the Lamperti transform
  $B_{H,K}^{\mathrm{LT}}(t)=e^{-\alpha HK t}B_{H,K}(e^{\alpha t})$.}
  \label{fig:traj_bifBm}
\end{figure}

We now turn to ergodicity. For any stationary process $X^{\mathrm{LT}}$,
Birkhoff's theorem ensures that, almost surely,
\[
\frac{1}{T}\int_{0}^{T}(X^{\mathrm{LT}}(s))^{2}\,ds
\;\xrightarrow[T\to\infty]{}\;
\mathbb{E}\big[(X^{\mathrm{LT}}(0))^{2}\big].
\]
For the present models these limits are
\[
\mathbb{E}\big[(S_{H}^{\mathrm{LT}}(0))^{2}\big]=2-2^{2H-1},
\qquad
\mathbb{E}\big[(B_{H,K}^{\mathrm{LT}}(0))^{2}\big]=1.
\]
The left panel of Figure~\ref{fig:subfBm_TA_var_CF} shows the empirical convergence of the second moment of $S_{H}^{\mathrm{LT}}$, stabilising around $2-2^{2H-1}$ as expected. The rapid convergence reflects the exponential covariance decay established in Lemma~\ref{lem:asymp_subfbm_scaled} and Theorem~\ref{thm:ergodic_scaled_subfbm}.

To assess Gaussianity, we compare empirical characteristic functions with their theoretical forms. For the sub-fractional Lamperti process we have
\[
\mathbb{E}\big[e^{ikS_{H}^{\mathrm{LT}}(t)}\big]
=\exp\!\left(-\tfrac{1}{2}k^{2}\big(2-2^{2H-1}\big)\right),
\]
while for the bi-fractional Lamperti process,
\[
\mathbb{E}[e^{ikB_{H,K}^{\mathrm{LT}}(t)}]
=\exp\!\left(-\tfrac{1}{2}k^{2}\right).
\]
The right panel of Figure~\ref{fig:subfBm_TA_var_CF} shows the empirical characteristic function of $S_H^{\mathrm{LT}}$, which closely follows the Gaussian prediction.

\begin{figure}[H]
  \centering
  \includegraphics[width=\textwidth]{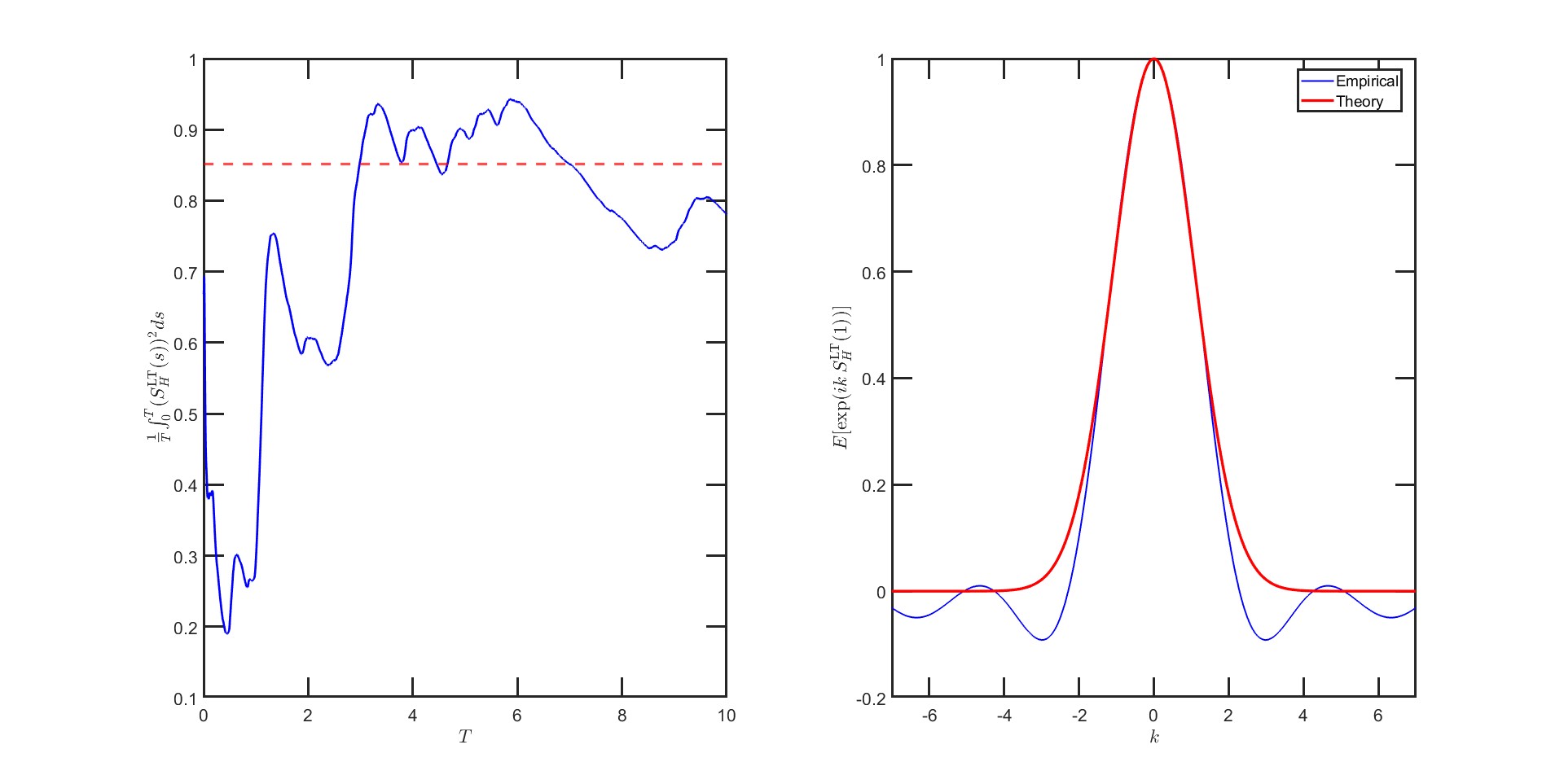}
  \caption{Lamperti $s-fBm$  with $H=0.6$, $\alpha=3$.
  Left: empirical long time second moment approaching its theoretical limit
  $2 - 2^{2H-1}$. Right: empirical characteristic function compared with the Gaussian characteristic function.}
  \label{fig:subfBm_TA_var_CF}
\end{figure}

Finally, Figure~\ref{fig:bifBm_CF} shows the corresponding characteristic function for the bi fractional case, matching the standard Gaussian form. This confirms that the Lamperti transforms are Gaussian with variances consistent with the analytical results.

\begin{figure}[H]
  \centering
  \includegraphics[width=\textwidth]{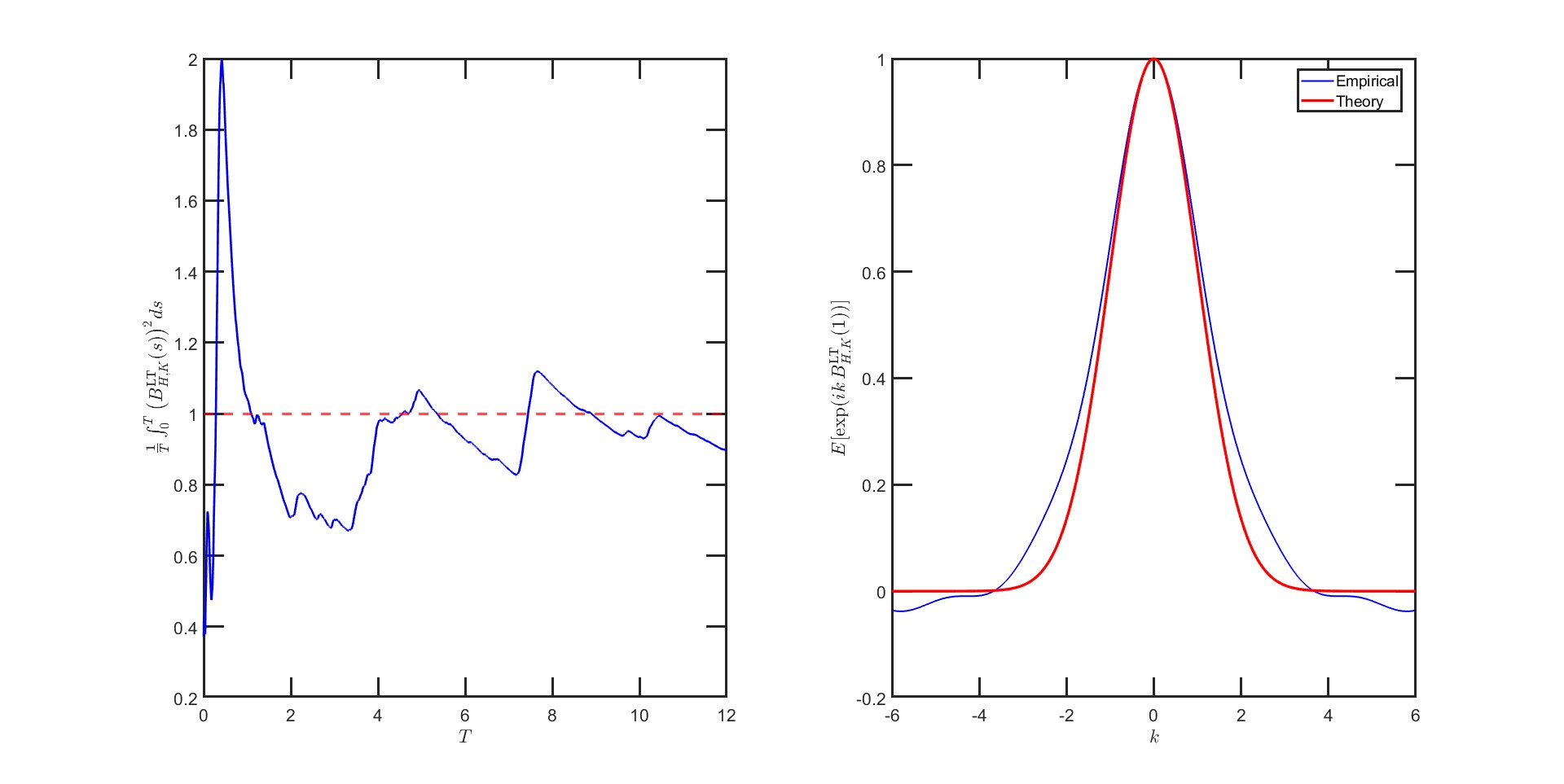}
  \caption{Lamperti $bi-fBm$  with $H=0.8$, $K=0.6$, $\alpha=1.5$.
  Empirical characteristic function compared with the
  standard Gaussian characteristic function.}
  \label{fig:bifBm_CF}
\end{figure}

\section{Conclusion}
\label{sec:conclusion}

We have shown that scaled Lamperti transforms convert $s$-fBm, $bi$-fBm, and
their Langevin equation integral variants into stationary Gaussian processes with
exponentially decaying correlations, yielding full ergodicity and strong mixing
in contrast to the slower decay produced by the classical Lamperti transform.
The explicit covariance representations allow mixing rates to be quantified in
terms of $(H, K, \alpha)$ and justify single trajectory reconstruction of
ensemble properties for these non-stationary models. More broadly, the results
clarify the role of the scaling parameter $\alpha$ in enforcing exponential
decorrelation and reveal how the interaction between $H$ and $K$ in the
bi-fractional case leads to distinct exponential regimes, underscoring the
usefulness of scaled Lamperti methods for analysing self-similar processes with
non-stationary increments.

\section*{Statements and Declarations}

\subsection*{Funding}
No funding was received to assist with the preparation of this manuscript.

\subsection*{Competing Interests}
The authors have no competing interests to declare that are relevant to the content of this article.

\subsection*{Data Availability}
No datasets were generated or analysed during the current study. Simulation code is available from the corresponding author upon reasonable request.

\end{document}